\documentclass{article}

\usepackage{lipsum}
\usepackage{amsfonts}
\usepackage{graphicx}
\usepackage{amsmath}
\usepackage{xcolor}
\usepackage{bm}
\usepackage{mathabx}
\usepackage{algorithm}
\usepackage{algpseudocode}
\usepackage{amsthm}
\usepackage{geometry}
\newcommand{\bo}[1]{\mathbf{#1}}
\newcommand{\bb}[1]{\mathbb{#1}}
\newcommand{\kk}[1]{\mathfrak{#1}}
\newcommand{\cc}[1]{\mathcal{#1}}
\newcommand{\nc}{\newcommand}
\nc{\cblue}[1]{\color{blue}{#1}\color{black}{}}
\nc{\cred}[1]{\color{red}{#1}\color{black}{}}
\nc{\cgrey}[1]{\color{gray}{#1}\color{black}{}}
\nc{\cmagenta}[1]{\color{magenta}{#1}\color{black}{}}
\nc{\corange}[1]{\color{orange}{#1}\color{black}{}}
\nc{\cviolet}[1]{\color{violet}{#1}\color{black}{}}
\nc{\kp}[1]{K^{(#1)}}
\nc{\mat}[1]{\begin{bmatrix}#1\end{bmatrix}}
\nc{\diag}[2]{\displaystyle\mathop{diag}_{#1}\left(#2\right)}
\newtheorem{theorem}{Theorem}
\newtheorem{remark}{Remark}

\title{Kroneckerised Particle Mesh Ewald}

\author{Igor Chollet\footnote{LAGA, Université Sorbonne Paris Nord, UMR 7539, Villetaneuse, France, chollet@math.univ-paris13.fr}}

\begin{document}

\maketitle

\begin{abstract}    
    Particle Mesh Ewald (PME) methods accelerated through Fast Fourier Transforms (FFTs) for their reciprocal part are widely used to solve $N$-body problems over periodic structures with Laplace-like kernels. The FFT dependence of classical PME may mitigate its performance on parallel distributed-memory architectures. We here introduce a new variant of the reciprocal part based on Sum of Kronecker Products (SKP) instead of FFT. Moreover, our implementation of this new method is not linearithmic (as opposed to classical PME) but has an important parallel potential. We present the different approximation levels exploited in our new scheme and demonstrate to what extent it could be used on parallel distributed-memory architectures. Numerical examples supplement presented assertions. 
    
    \end{abstract}

\section{Introduction}
This paper focuses on the parallel solving of periodic N-body problems, which are particularly interesting for molecular dynamics models. More generally, the class of N-body problems considered in this paper corresponds to those studied in Ewald summations. 

\subsection{Periodic $N$-body problems and notations}
For any finite sets $\bb{X}$, $\bb{Y}$ (whose elements are named \textit{particles}) and for $\bb{K} = \bb{R}$ or $\bb{C}$, we denote by $\bb{K}[\bb{X}]$ the set of applications from $\bb{X}$ to $\bb{K}$ and by $\bb{K}[\bb{X},\bb{Y}]$ the set of linear operators from $\bb{K}[\bb{Y}]$ to $\bb{K}[\bb{X}]$. In our application case, let $\bb{X}$ and $\bb{Y}$ be two point clouds lying in a box $B$ of radius $r_B$ centered in $\bo{c}_B$ and let $\bo{q}\in \bb{C}[\bb{Y}]$. We use $\bo{q}=[q(\bo{y})]_{\bo{y}\in \bb{Y}}$ to denote a column vector with entries associated to elements of $\bb{Y}$. We consider the task of solving the following $N$-body problem
\begin{equation}
\label{eq::qmutheta_coulomb}
    \forall\hspace{0.05cm}\bo{x}\in \bb{X},\hspace{0.8cm}p(\bo{x}) = \sum_{\bo{t}\in 2r_B\bb{Z}^3}\sum_{\bo{y}\in \bb{Y}}G(|\bo{x}-\bo{y}+\bo{t}|)q(\bo{y}),
\end{equation}
i.e. the problem of computing $\bo{p}\in \bb{C}[\bb{X}]$ from $\bo{q}$ applying $\big[\sum_{\bo{t}\in 2r_B\bb{Z}^3}G(|\bo{x}-\bo{y}+\bo{t}|)\big]_{\bo{x}\in\bb{X},\bo{y}\in\bb{Y}}\in \bb{C}[\bb{X},\bb{Y}]$ (notation for matrix with rows indexed by elements of $\bb{X}$ and columns indexed by elements of $\bb{Y}$) represented by a matrix of nodal evaluation of a Green-like kernel $G$. In other words, it consists in computing every possible pairwise interactions between particles of target point cloud $\bb{X}$ and particles of source point cloud $\bb{Y}$ with respect to a (here periodic) interaction law. For the sake of clarity, we will consider that $G(p) = R^{-1}$ denotes the Coulombic kernel. The main issue in directly considering Eq. \eqref{eq::qmutheta_coulomb} lies in the \textit{conditionally} convergent behavior of the series over $\bb{Z}^3$, possibly prohibiting accurate naive computation on large point clouds. Moreover, even when truncating this series, the problem still has a quadratic complexity bound due to a number of evaluation of $G$ proportional to the product sizes of point clouds $\bb{X}$ and $\bb{Y}$, supposed to be $N$ in this article.

Such problem arises in many scientific areas, especially in molecular dynamics simulations in which particles ($\bo{x}$'s and $\bo{y}$'s) are atoms and $\bo{q}$ corresponds to a vector of atom charges. Moreover, the quadratic cost is too prohibitive for modern applications. Fortunately, there exist several fast approaches to numerically deal with Eq. \eqref{eq::qmutheta_coulomb}. Among them, Particle-Mesh Ewald, denoted PME \cite{pmeref,spmeref,TOUKMAJI199673}; Fast Multipole Methods \cite{ankh,challacombe,407723,YAN2018214,FONG20098712,GREENGARD1987325,edarig}; or recently Dual-space Multilevel Kernel splitting framework \cite{https://doi.org/10.1002/cpa.22240}.

In the following, vectors will be written using \textbf{bold letters} and their elements are referred to through the same letter without bold and indexed by elements of the set they are defined on (e.g. $\bo{q}=(q(\bo{y}_0),...,q(\bo{y}_{N-1}))^T$ for $\bo{q}\in\bb{C}[\bb{Y}]$ with $\bb{Y}=\{\bo{y}_0,...,\bo{y}_{N-1}\}$ for a point cloud with $N$ points). Notice that the use of freely generated spaces allows us to avoid the cumbersome writing style required by the ordering of point cloud elements. Linear operators (i.e. matrices most of the time in this article) on freely generated vector spaces will be denoted by calligraphic letters (e.g. $\cc{A}$). Finite sets, point clouds or grids are denoted using blackboard bold letters (e.g. $\bb{Y}$). The euclidean norm shall be denoted by $|\cdot |$ and $|\cdot |_\infty$ will refer to the maximum norm.

\subsection{Particle Mesh Ewald}
\label{s:pme}
Particle Mesh Ewald (PME) \cite{pmeref,spmeref} is a linearithmic (i.e. $\cc{O}(NlogN)$ cost) method allowing to deal with Eq. \eqref{eq::qmutheta_coulomb} in a mathematically convergent and numerically stable way, built upon Ewald Summation techniques \cite{TOUKMAJI199673}, following Ewald's method \cite{ewald_german}. This approach relies on a splitting between long and short ranges, the first treated in the \textit{real space} and the second in the Fourier (named \textit{reciprocal}) space, classically allowing to exploit FFTs. The method can be summarised as follows: let $\tau\hspace{0.03cm}:\hspace{0.03cm}\bb{R}^{+*}\rightarrow \bb{R}$ be a truncation function such that $1- \tau(p)$ quickly tends to $0$ and $ \tau$ is \textit{sufficiently} smooth. We have
\begin{equation}
\label{eq::tau}
    G(p) = \underbrace{(1- \tau(p))G(p)}_{\text{short range}} + \underbrace{ \tau(p)G(p)}_{\text{long range}}.
\end{equation}
When carefully choosing $ \tau$, one can express the series over the short range part as an absolutely convergent series and the one over the long range part as follows
\begin{equation}
    \sum_{\bo{t}\in 2r_B\bb{Z}^3}\sum_{\bo{y}\in \bb{Y}} \tau(|\bo{x}-\bo{y}+\bo{t}|)G(|\bo{x}-\bo{y}+\bo{t}|)q(\bo{y}) =: \sum_{\bo{y}\in \bb{Y}}\cc{H}(\bo{x},\bo{y})q(\bo{y}),
\end{equation}
$\cc{H}$ being computed with the Fourier transform of the long range part. For the Coulombic kernel and the error function $ \tau(p) = erf(\xi R)$, $\cc{H}$ has the following expression
\begin{equation}
\label{eq::hexpr}
    \cc{H}(\bo{x},\bo{y})=\sum_{0\neq \bo{m}\in \bb{Z}^3}\underbrace{\frac{e^{-\frac{\pi ^2}{\xi^2}|\bo{m}|^2}}{|\bo{m}|^2}}_{=:\alpha(|\bo{m}|)}e^{2i\pi \langle \tilde{\bo{x}}-\tilde{\bo{y}},\bo{m}\rangle}
\end{equation}
where $\xi$ denotes the Ewald parameter and $\tilde{\bo{x}} := \frac{\bo{x}-\bo{c}_B}{r_B}$ refers to the normalised relative position of $\bo{x}$ with respect to the computational box (without loss of generality, we assume here that the box is cubic). The series in Eq. \eqref{eq::hexpr} can be truncated (as well as the series in the real space for the short range part). The larger $\xi$, the larger the number of Fourier modes must be considered in Eq. \eqref{eq::hexpr} for a target output accuracy. By optimizing $\xi$, one may obtain the Ewald Summation resulting in $\cc{O}(N^{3/2})$ complexity \cite{simonett}. The PME alternative consists in choosing $\xi$ such that most of the computation takes place in the Fourier space (that is in the long range part). Then, Eq. \eqref{eq::hexpr} is computed through the factorisation
\begin{equation*}
    \sum_{\bo{y}\in \bb{Y}}\cc{H}(\bo{x},\bo{y})q(\bo{y}) \approx \sum_{0\neq \bo{m}\in \bb{Z}^3}\alpha(|\bo{m}|)e^{2i\pi \langle \tilde{\bo{x}},\bo{m}\rangle}\underbrace{\sum_{\bo{y}\in Y}e^{-2i\pi \langle\tilde{\bo{y}},\bo{m}\rangle}q(\bo{y})}_{\textit{structure factor}}
\end{equation*}
where the structure factor can be practically assembled through a FFT over vector $\bo{q}\in\bb{R}[\bb{Y}]$ (after an interpolation of particle positions over a cartesian grid and a truncation over Fourier modes). 
The FFT linearithmic complexity thus dominates the computation. To be more precise, using our notations, PME aims at computing
\begin{equation}
\label{eq::sumhexpr}
    \cc{H}_M\cdot\bo{q}, \hspace{0.5cm}\cc{H}_M(\bo{x},\bo{y}) = \sum_{0\neq \bo{m}\in \bb{M}^3}\alpha(|\bo{m}|)e^{2i\pi \langle \tilde{\bo{x}}-\tilde{\bo{y}},\bo{m}\rangle},
\end{equation}
with $\bb{M} = [\![-M,M]\!]$, $\cc{H}_M\in \bb{C}[\bb{X},\bb{Y}]$ and $\bo{q}\in \bb{C}[\bb{Y}]$, using FFTs. 

\subsection{Problem, contributions and article structure}
Despite the massive use of PME in molecular dynamics simulation codes thanks to high performance of modern FFT implementations (such as \cite{FFTW05}), its performance directly depends on this last, that may mitigate on distributed memory architectures \cite{10.1007/978-3-030-86359-3_21}. However, efficient handling of such architectures is a crucial need in order to treat large scale simulations. In addition, PME algorithm does not particularly benefit from the multivariate nature of Fourier matrices arising in its fast scheme. In the present article, we wonder to what extent this could be exploited in order to find efficient evaluation schemes, especially on distributed memory architectures. We answer this question through
\begin{itemize}
    \item providing a new formulation of the PME reciprocal part as a Sum of Kronecker Products, presented in Sect. \ref{s::sumofkronprod}, denoted as KPME;
    \item exploiting parallel Kronecker Products for KPME, presented in Sect. \ref{s::fkpeog};
    \item proposing a novel scalable scheme to evaluate KPME on distributed memory architectures, given in Sect. \ref{s::kpmepar}, based on a \textit{divide-and-conquer} strategy.
\end{itemize}
Theoretical error bounds and complexities are proposed. Numerical experiments complement our theoretical claims and are provided in Sect. \ref{s::numericals}. For the sake of reproducibility, and to cover implementation details that cannot be mentioned in this article, the code used for testing and developed for this paper is also provided\footnote{https://github.com/IChollet/kpme}.

\section{Sum of Kronecker product}
\label{s::sumofkronprod}
In this section, we demonstrate that the matrix form of the long range part of the Ewald Summation can be approximated as a sum of weighted Kronecker products. The global result is summarised in Thm. \ref{thm::kronpme}. This theorem claims that, for any targeted precision $\epsilon$, one may find an expression of the reciprocal part of Ewald summation in the form of a Sum of Kronecker Product (SKP), at the cost of passing through a cartesian grid, up to a controllable error $\epsilon$. Thus, due to this particle mapping over a grid, this procedure enters the category of \textit{particle-mesh}. We recall that the Kronecker product $\otimes$ between two matrices denotes the matrix tensor product using the standard basis, that is
\begin{equation}
    \begin{bmatrix}a_{00}\hdots a_{0N}\\\vdots\ddots\vdots\\a_{M0}\hdots a_{MN}\end{bmatrix}\otimes\cc{B} := \begin{bmatrix}a_{00}\cc{B}\hdots a_{0N}\cc{B}\\\vdots\ddots\vdots\\a_{M0}\cc{B}\hdots a_{MN}\cc{B}\end{bmatrix}.
\end{equation}

\begin{theorem}[Existence of KPME]
\label{thm::kronpme}
Let $M\in \bb{N}^*$, let $\cc{H}_M$ be the reciprocal kernel defined in Eq. \eqref{eq::sumhexpr} and let $\epsilon > 0$. Then there exist $ \Xi_\bb{Y}\subset \bb{R}^{3}$ (resp. $ \Xi_\bb{X}$) a product of finite grids, $ \Xi_\bb{Y} =  \Xi_{\bb{Y},0}\otimes  \Xi_{\bb{Y},1} \otimes  \Xi_{\bb{Y},2}$ (resp. $ \Xi_\bb{X} =  \Xi_{\bb{X},0}\otimes  \Xi_{\bb{X},1} \otimes  \Xi_{\bb{X},2}$), there exist a finite set $\big\{ \cc{A}^{(p)}(\lambda)\in \bb{R}[ \Xi_\bb{X}, \Xi_\bb{Y}]|k\in \{0,1,2\}\big\}_{\lambda\in \Lambda}$ of linear operators parameterised by elements of a finite set $\Lambda$, there exist $\cc{S}_\bb{Y}\in \bb{R}[ \Xi_\bb{Y},\bb{Y}]$ and $\cc{S}_\bb{X}^*\in \bb{R}[\bb{X}, \Xi_\bb{X}]$ such that if the computational box $B$ has a smaller radius than $M^{-1}$, then
\begin{equation}
    \bigg|\bigg|\cc{H}_M-\cc{S}_\bb{X}^*\left(\sum_{\lambda\in \Lambda}\left(\cc{A}^{(0)}(\lambda)\otimes \cc{A}^{(1)}(\lambda)\otimes \cc{A}^{(2)}(\lambda)\right)\right)\cc{S}_\bb{Y}-c_\Lambda\bo{1}_N\bo{1}_N^T\bigg|\bigg| < \epsilon,
\end{equation}
where the cardinal of $\Lambda$ denoted by $|\Lambda|$ verifies $|\Lambda| \leq (2M+1)^2$, $\bo{1}_{N}$ denoting a vector of length $N$ full of ones and $c_\Lambda$ is a constant.
\end{theorem}
The proof of Thm. \ref{thm::kronpme} is based on three points: a decomposition (or approximation) of the vector of nodal evaluations of $\alpha$ as a sum of Kronecker products (presented in Sect. \ref{ss:decompalpha}), an interpolation of particle positions over product grids (presented in Sect. \ref{ss::multpolintprod}), a use of actions of the underlying invariant group (presented in Sect. \ref{ss::symfourier}). Finally, in Sect. \ref{ss::overallcomp}, we describe how these elements combine to obtain Thm. \ref{thm::kronpme}.

\subsection{Decomposition of $\alpha_M$}
\label{ss:decompalpha} In this section, we aim at expressing $\bm{\alpha}_M$, that is the vector
\begin{equation}
    \bm{\alpha}_M := \Bigg[\begin{cases}\alpha(|\bo{m}|)\text{ if }\bo{m}\neq \bo{0}\\0\text{ if }\bo{m}=\bo{0}\end{cases}\Bigg]_{\bo{m}\in  \bb{M}^3}
\end{equation}
as a SKP, where $\bb{M} := [\![-M,M]\!]$. The problem we are trying to solve is thus the following: given $\epsilon$, find $\Lambda$ a finite set, $\omega_\lambda$, $c_\Lambda\in \bb{R}$, $\bm{\alpha}_{0,\lambda}$, $\bm{\alpha}_{1,\lambda}$, $\bm{\alpha}_{2,\lambda}\in \bb{R}\big[\bb{M}\big]$ such that
\begin{equation}
\label{eq::pb_skp}
    \bigg|\bigg|\bm{\alpha}_M - \sum_{\lambda\in \Lambda}\omega_\lambda \left(\bm{\alpha}_{0,\lambda}\otimes \bm{\alpha}_{1,\lambda}\otimes \bm{\alpha}_{2,\lambda}\right)+c_\Lambda\bm{\delta}\bigg|\bigg| < \epsilon,
\end{equation}
$\bm{\delta}$ being a given vector used to correct the result. This can be achieved in multiple ways. We here present two of them: one suited for general choice (i.e. \textit{kernel-independent}) of function $\tau$ (see Sect. \ref{s:pme}) is described in Sect. \ref{sss::nkpa}, a second being adapted to the usual choice $\tau(p)= erf(p)$ is given in Sect. \ref{sss::quadrature_based}.

\subsubsection{Nearest Kronecker Product Approximation}
\label{sss::nkpa}
When $c_\Lambda=0$, the problem \eqref{eq::pb_skp} becomes a particular case of nearest Kronecker product approximation \cite{LOAN200085} consisting for a matrix $\cc{A}$ of finding $\cc{B}$ and $\cc{C}$ such that $\cc{A}\approx \cc{B}\otimes \cc{C}$. It can be solved in a simple way by exploiting a sequence of Singular Value Decomposition (SVD). Indeed, if $\cc{A}= \cc{U}\Sigma \cc{V}^T = \sum_{k}\sigma_k\bo{u}_k\bo{v}_k^T= \sum_{k}\sigma_k\bo{u}_k\otimes \bo{v}_k^T$, where $\bo{u}_k$ (resp. $\bo{v}_k$) denotes the $k^{th}$ column of the orthonormal matrix $\cc{U}$ (resp. $\cc{V}$), we have a SKP formulation of matrix $\cc{A}$.

Let $\cc{I}_2^{\circ}\hspace{0.1cm}:\hspace{0.1cm}\bb{K}[\Xi_0\times \Xi_1\times \Xi_2]\rightarrow \bb{K}[\Xi_2, \Xi_0\times \Xi_1]$ that \textit{matricises} a vector whose entries are defined on a product grid in the following way: the output matrix has rows corresponding to elements of $\Xi_2$ and columns corresponding to elements of $\Xi_0\times \Xi_1$. We denote by $\cc{I}_2^{\bullet}$ the reverse operation that transforms such a matrix into a vector whose entries correspond to elements of $\Xi_0\times \Xi_1\times \Xi_2$. In a more general way, we define by $\cc{I}_{k}^{\circ}\hspace{0.1cm}:\hspace{0.1cm}\bb{K}[\Xi_0\times \hdots\times \Xi_{d-1}]\rightarrow \bb{K}[\Xi_k, \Xi_0\times\hdots\times\Xi_{k-1}\times \Xi_{k+1} \times \hdots\Xi_{d-1}]$ the reshape of a vector into a matrix with rows corresponding to elements of $\Xi_k$.

By means of singular value decomposition, one obtains:
\begin{equation}
    \cc{I}_2^{\circ}(\bm{\alpha}_M) = \cc{U}\Sigma \cc{V}^T = \sum_{i}\sigma_i \bo{u}_i\otimes \bo{v}_i^T
\end{equation}
where $\bo{u}_i\in \bb{C}\big[\bb{M}\big]$, $\bo{v}_i\in \bb{C}\big[\bb{M}^2\big]$. Then, we simply recover a vector using:
\begin{equation}
    \bm{\alpha}_M = \cc{I}_2^{\bullet}\left(\cc{I}_2^{\circ}(\bm{\alpha}_M)\right) = \sum_{i}\sigma_i \bo{u}_i\otimes \bo{v}_i.
\end{equation}

This process can be repeated recursively on each $\bo{v}_i$ in order to directly build a full SKP approximation. This gives
\begin{equation}
\label{eq::last_decomp_kron}
    \cc{I}_1^{\circ}(\bo{v}_i) = \underbrace{\cc{U}_i\Sigma_i\cc{V}_i^T}_{\text{SVD decomposition}} = \sum_{j}\sigma_{i,j} \underbrace{\bo{u}_{i,j}}_{\substack{\text{column }j\\\text{ of }\cc{U}_i}}\otimes \bo{v}_{i,j}^T\Rightarrow \bm{\alpha}_M = \sum_{i,j}\sigma_i\sigma_{i,j}\bo{u}_i\otimes \bo{u}_{i,j}\otimes \bo{v}_{i,j}.
\end{equation}
It is well known \cite{doi:10.1137/1.9781421407944} that the SVD can be used in order to obtain low-rank approximations. Since we have that
$||\cc{I}_k^{\circ}(\bm{\alpha})||_F = ||\bm{\alpha}||_F$, where $||\cdot ||_F$ denotes the Frobenius norm, the classical estimate for SVD (Eckart-Young-Mirsky theorem) gives us
\begin{equation}
\label{eq::eym}
\begin{aligned}
        \bigg|\bigg|\bm{\alpha}_M-\underbrace{\sum_{j=0}^{r}\sigma_{i}\bo{u}_{i}\otimes \bo{v}_{i}}_{\mu_i}\bigg|\bigg|_F &= \bigg|\bigg|\cc{I}_2^{\circ}(\bm{\alpha}_M)-\sum_{i=0}^{r}\sigma_{i}\bo{u}_{i}\otimes \bo{v}_{i}^T\bigg|\bigg|_F = \sqrt{\sum_{i>r}\sigma_i^2}.
\end{aligned}
\end{equation}
We thus obtain the following estimate:
\begin{equation}
\begin{aligned}
        \bigg|\bigg|\bm{\alpha}_M-\sum_{i,j}\sigma_i\sigma_{i,j}\bo{u}_i\otimes \bo{u}_{i,j}\otimes \bo{v}_{i,j}\bigg|\bigg|_F &= \bigg|\bigg|\bm{\alpha}_M-\mu_i+\mu_i-\sum_{i,j}\sigma_i\sigma_{i,j}\bo{u}_i\otimes \bo{u}_{i,j}\otimes \bo{v}_{i,j}\bigg|\bigg|_F\\
        &\leq \underbrace{\bigg|\bigg|\bm{\alpha}_M-\mu_i\bigg|\bigg|}_{\text{Eq.\eqref{eq::eym}}}+\underbrace{\bigg|\bigg|\mu_i-\sum_{i,j}\sigma_i\sigma_{i,j}\bo{u}_i\otimes \bo{u}_{i,j}\otimes \bo{v}_{i,j}\bigg|\bigg|}_{=:\mu_{i,j}}
\end{aligned}
\end{equation}
in which
\begin{equation}
    \begin{aligned}
        \mu_{i,j} &= \bigg|\bigg|\sum_{i=0}^{r_0}\sigma_i\bo{u}_i\otimes\left(\bo{v}_i-\sum_{j=0}^{r_1}\sigma_{i,j}\bo{u}_{i,j}\bo{v}_{i,j}\right)\bigg|\bigg|_F\\
        &\leq \sum_{i=0}^{r_0}\sigma_i\underbrace{||\bo{u}_i||_F}_{=1}\underbrace{\bigg|\bigg|\cc{I}_1^{\circ}(\bo{v}_i)-\sum_{j=0}^{r_1}\sigma_{i,j}\bo{u}_{i,j}\bo{v}_{i,j}^T\bigg|\bigg|_F}_{=\sqrt{\sum_{j>r_1}\sigma_{i,j}^2}\text{ by Eckart-Young-Mirsky Thm.}}.
    \end{aligned}
\end{equation}
The final estimate thus writes
\begin{equation}
\label{eq::estimate_svd}
        \bigg|\bigg|\bm{\alpha}_M-\sum_{i,j}\sigma_i\sigma_{i,j}\bo{u}_i\otimes \bo{u}_{i,j}\otimes \bo{v}_{i,j}\bigg|\bigg|_F \leq \left(\sqrt{\sum_{i>r_0}\sigma_i^2}\right)+\left(\sum_{i=0}^{r_0}\sigma_i\sqrt{\sum_{j>r_1}\sigma_{i,j}^2}\right)
\end{equation}
where $r_0$ and $r_1$ are the truncation ranks chosen in the different SVDs. Moreover, provided that the singular values of $\cc{I}_2^{\circ}(\bm{\alpha}_M)$ quickly decay, it is possible to efficiently reduce the number of SKP terms. Notice that Eq. \eqref{eq::estimate_svd} leads us to choose a smaller tolerance in the truncated SVD over $\cc{I}_1^{\circ}(\bo{v}_i)$'s than over $\cc{I}_2^{\circ}(\bm{\alpha}_M)$.

According to the notation in Eq. \eqref{eq::pb_skp}, in the case of SVD-based SKP approximation of $\bm{\alpha}$, we thus came up with $\Lambda = [\![0,r_0]\!]\times [\![0,r_1]\!]$, $\omega_{(i,j)} = \sigma_i\sigma_{i,j}$, $c_\Lambda = 0$, $\bm{\alpha}_{0,(i,j)}=\bo{u}_i$, $\bm{\alpha}_{1,(i,j)}=\bo{u}_{i,j}$ and $\bm{\alpha}_{2,(i,j)}=\bo{v}_{i,j}$. In Fig. \ref{fig::ranks_kron}, we depicted how this process effectively compress the SKP approximation of $\bm{\alpha}$ depending on $\xi$ in the particular setting $\tau(p) = erf(p)$.
\begin{figure}
    \centering
    \includegraphics[width=0.8\linewidth]{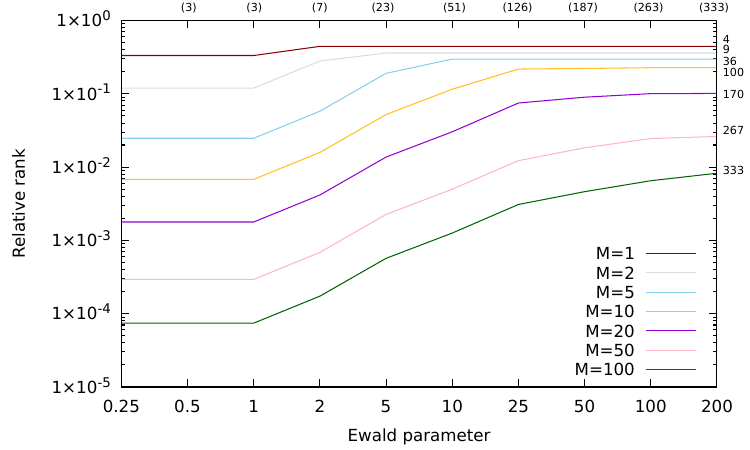}
    \caption{\label{fig::ranks_kron}Relative ranks (i.e. $\frac{r}{(2M+1)^2}$ where $r$ is the number of terms in the SKP approximation) of SKP for $\bm{\alpha}_M$ w.r.t. various values of \textit{Ewald parameter} $\xi$ for $\tau(p) = erf(\xi R)$ and $M=\xi$ with SVD tolerance $10^{-8}$. In top and parentheses are indicated the maximal rank over all tested $M$ for each corresponding $\xi$ and on the right are given the final ranks for $\xi=100$ for each $M$.}
\end{figure}

\begin{remark}
   The choice of putting $\bo{0}$ as value in $\bm{\alpha}_M$ for unused Fourier mode $\bo{m}=\bo{0}$ actually may not be optimal: one may also consider a better alternative in order to minimise the numerical ranks. However, we found in our tests that this choice of zero only had a small impact on these ranks.
\end{remark}
This algebraic way of obtaining the SKP has the advantage of not depending on the truncation function $\tau$. In particular cases, one may provide more efficient compression, as we are about to present it.

\subsubsection{Quadrature-based expansion}
\label{sss::quadrature_based}
In the particular case $\tau(p) = erf(\xi R)$, the explicit expression of $\bm{\alpha}_M$ is known through Eq. \eqref{eq::hexpr}:
\begin{equation}
    \bm{\alpha}_M := \Bigg[\frac{e^{-\frac{\pi ^2}{\xi^2}|\bo{m}|^2}}{|\bo{m}|^2}\Bigg]_{\bo{m}\in  \bb{M}^3} = \underbrace{\bigg[|\bo{m}|^{-2}\bigg]_{\bo{m}\in  \bb{M}^3}}_{=:\bm{\beta}_M}\odot\bigg[e^{-\frac{\pi ^2}{\xi^2}|\bo{m}|^2}\bigg]_{\bo{m}\in  \bb{M}^3}
\end{equation}
where $\odot$ denotes the Hadamard (i.e. element-wise) vector product. Since $e^{-\frac{\pi ^2}{\xi^2}|\bo{m}|^2} = e^{-\frac{\pi ^2}{\xi^2}m_0^2}e^{-\frac{\pi ^2}{\xi^2}m_1^2}e^{-\frac{\pi ^2}{\xi^2}m_2^2}$, it is clear that only $\bm{\beta}_M$ prohibits the Kronecker expression. We thus want to express $\bm{\beta}_M$ as a SKP itself. To do so, we exploit the quadratures of the integral formula for inverse function over $\bb{R}^*$ \cite{inversedistancequad}:
\begin{equation}
\label{eq::one_over_r_int}
    \frac{1}{R} = \int_{\bb{R}^+}e^{-tR}dt \approx \sum_{\lambda\in \Lambda}\omega_\lambda e^{-\lambda R}.
\end{equation}
Efficient quadrature (with nodes $\lambda$'s and associated weights $\omega_\lambda$) of this integral can be obtained in the case $R \geq 1$ \cite{inversedistancequad,yarvinrokhlin}. This is convenient since $|\bo{m}|_\infty\geq 1$ for any $\bo{m}\in \bb{M}^3\backslash \{\bo{0}\}$ and the zero-case is equal to zero. Among them, the \textit{cardinal sine} quadrature with weights $\omega_{\lambda_l}$'s and nodes $\lambda_l$'s writes \cite{braesshackbusch}
\begin{equation}
    \omega_{\lambda_l} := \frac{h}{1+e^{-lh}}, \hspace{0.2cm}\lambda_l := log(1+e^{lh}),\hspace{0.2cm}l\in [\![-N_{quad},N_{quad}]\!],
\end{equation}
where $h$ is a \textit{lifting factor} (asymptotically chosen as $h=\pi/\sqrt{2N_{quad}+1}$) and $N_{quad}$ is an integer. This rule is clearly non-optimal in terms of number of nodes for target precision but has the advantage of being given by explicit formula, hence directly obtained without solving any system. 
However, the best choice of $h$ may not be the asymptotic one, thus this $h$ is found by testing \textit{in a brute force way} different values of $h$ for each $N$, checking the error on possible arguments (that are $|\bo{m}|^2$, $\bo{m}\in \bb{M}^3\backslash \{\bo{0}\}$) and keeping the best each time in our code. Denoting $\Lambda$ this set of quadrature nodes, it follows that 
\begin{equation}
    \bm{\beta}_M = \bigg[\frac{1}{|\bo{m}|^2}\bigg]_{\bo{m}\in  \bb{M}^3}\approx \sum_{\lambda\in \Lambda}\omega_\lambda \left(\bigg[e^{-\lambda (m_0^2+m_1^2+m_2^2)}\bigg]_{\bo{m}\in  \bb{M}^3} - \underbrace{\big[\delta_{\bo{0}}(\bo{m})\big]_{\bo{m}\in  \bb{M}^3}}_{=:\bm{\delta}}\right),
\end{equation}
where $\delta_{\bo{0}}(\bo{m}) = \begin{cases}1\text{ if }\bo{m} = \bo{0}\\0\text{ otherwise}\end{cases}$. This implies that 
\begin{equation}
\label{eq::quaddecompalpham}
    \bm{\alpha}_M \approx \left(\sum_{\lambda\in \Lambda}\omega_\lambda \bm{\alpha}_{\lambda}\otimes \bm{\alpha}_{\lambda}\otimes \bm{\alpha}_{\lambda}\right) - c_\Lambda\bm{\delta}, \hspace{0.3cm}\bm{\alpha}_{\lambda} := \bigg[e^{-\left(\lambda+\frac{\pi ^2}{\xi^2}\right)m^2}\bigg]_{m\in \bb{M}},
\end{equation}
where $c_\Lambda = \sum_{\lambda\in \Lambda} \omega_\lambda$. In Fig. \ref{fig::quadskp}, we depicted compression rates for sinc-quadrature-based SKP approximation of $\bm{\alpha}_M$ according to various precisions. One may notice that for a given target precision $\epsilon$, the exhaustive quadrature selection method we use (i.e. the brute force one testing each possible argument) ensures that $|\bm{\alpha}_M - \left(\left(\sum_{\lambda\in \Lambda}\omega_\lambda \bm{\alpha}_{\lambda,0}\otimes \bm{\alpha}_{\lambda,1}\otimes \bm{\alpha}_{\lambda,2}\right) + c_\Lambda\bm{\delta}\right)|_\infty < \epsilon$. 

\begin{figure}
    \centering
    \includegraphics[width=0.5\linewidth]{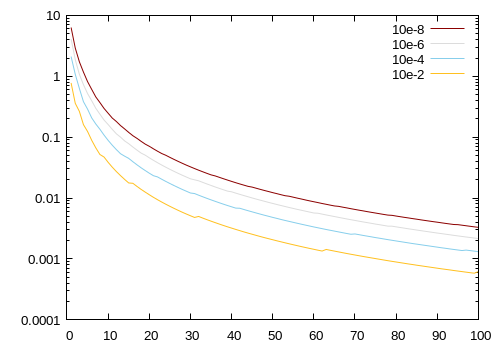}
    \caption{\label{fig::quadskp}For various $|\cdot |_\infty$ precisions in the cardinal sine quadrature (indicated with different colors) resulting in $N_{quad}$ points, compression rates are computed as $\frac{N_{quad}}{(2M+1)^2}$.}
\end{figure}

\begin{remark}
   The search of best quadrature for Eq. \eqref{eq::one_over_r_int} leads to design of optimal or quasi-optimal specific rules. Such search for a general maximal argument is out of the scope of this article. We still mention that tabulations available in the literature may help using quasi-optimal rules, for instance considering Yarvin-Rokhlin rules \cite{yarvinrokhlin}. Indeed, the rule with 27 points for precision $10^{-14}$ presented in this reference is suitable for $\bm{\alpha}_{12}$ and is (far) better than the best cardinal sine one found in our numerical experiments, using 279 nodes.
\end{remark}

Estimates for best quadrature of Eq. \eqref{eq::one_over_r_int} \cite{inversedistancequad} provide a bound for the optimal compression rate. Indeed, an upper bound of the error induced by optimal quadrature rule leads to 
\begin{equation}
    \bigg|\bm{\alpha}_M - \left(\sum_{\lambda\in \Lambda}\omega_\lambda \bm{\alpha}_{\lambda}\otimes \bm{\alpha}_{\lambda}\otimes \bm{\alpha}_{\lambda}\right) + c_\Lambda\bm{\delta}\bigg|_\infty \leq 16e^{-\pi\sqrt{|\Lambda^*|}},
\end{equation}
where $\Lambda^*$ is the index set for the \textit{best} quadrature. Hence, for a given precision $\epsilon > 0$,
\begin{equation}
\label{eq::estimatequad}
    16e^{-\pi\sqrt{|\Lambda^*|}} \leq \epsilon \Leftrightarrow |\Lambda^*|\geq \left(\frac{log\left(\frac{\epsilon}{16}\right)}{\pi}\right)^2.
\end{equation}
The cardinal $|\Lambda^*|$ being the number of terms in the SKP for optimal quadrature rules, a bound on the compression rate in this particular case is thus given by
\begin{equation}
    \bigg\lceil\frac{log\left(\frac{\epsilon}{16}\right)}{\pi}\bigg\rceil^2 (2M+1)^{-2}.
\end{equation}

\subsection{Multivariate polynomial interpolation on product grids}
\label{ss::multpolintprod}
Since the reciprocal kernel operates in the Fourier domain, one needs to switch between real space and the Fourier space. This can be done through the application of $\cc{F}_{\bb{M}^3,\bb{Y}}\in \bb{C}[\bb{M}^3,\bb{Y}]$, the Fourier matrix on particle distribution $\bb{Y}$, defined as
\begin{equation}
    \cc{F}_{\bb{M}^3,\bb{Y}}(\bo{m},\bo{y}) := e^{2i\pi \langle\bo{m},\bo{y}\rangle}.
\end{equation}
In the special case $\bb{Y}=\bb{Y}_0\times \bb{Y}_1\times \bb{Y}_2$, we directly have
\begin{equation}
\label{eq::targetexpformula}
    \cc{F}_{\bb{M}^3,\bb{Y}} = \underbrace{\bigg[e^{2i\pi my_0}\bigg]_{m \in \bb{M}, y_0\in \bb{Y}_0}}_{=:\cc{F}_{\bb{M},\bb{Y}_0}\in \bb{C}[\bb{M},\bb{Y}_0]}\otimes \underbrace{\bigg[e^{2i\pi my_1}\bigg]_{m \in \bb{M}, y_1\in \bb{Y}_1}}_{=:\cc{F}_{\bb{M},\bb{Y}_1}\in \bb{C}[\bb{M},\bb{Y}_1]}\otimes \underbrace{\bigg[e^{2i\pi my_2}\bigg]_{m \in \bb{M}, y_2\in \bb{Y}_2}}_{=:\cc{F}_{\bb{M},\bb{Y}_2}\in \bb{C}[\bb{M},\bb{Y}_2]}.
\end{equation}
In the general case, $\bb{Y}$ is not a product of one-dimensional grids. However, many fast numerical methods are based on interpolation of particle positions on multivariate cartesian grids defined as products of one-dimensional grids (see for instance Sect. \ref{s:pme}). This can be done through multivariate product Lagrange polynomial interpolation \cite{FONG20098712,ankh,pmeref} or use of $B$-splines \cite{spmeref}. We here consider the first option.

\subsubsection{General formula}
Let $2\leq L\in \bb{N}$, then for any set $\bb{L} := \{-1 = h_0^\diamond, \hspace{0.1cm}\hdots\hspace{0.1cm}, h_{L-1}^\diamond = 1\}$ of distinct points (the "diamond" notation points out that the points are interpolation nodes), we denote by $S_I(y^\diamond,y)$ the Lagrange interpolation polynomial over $I=[a,b]$ associated to $y^\diamond\in \bb{L}$ of the form
\begin{equation}
\label{eq::interp1d}
    S_I(y^\diamond,y) := \prod_{y^\diamond\neq z^\diamond\in\bb{L}}\frac{\phi_{I}(y)-z^\diamond}{y^\diamond-z^\diamond}, \hspace{0.2cm}\phi_{I}(y):=\frac{y-(b+a)/2}{(b-a)/2},\hspace{0.2cm}\forall y\in [a,b].
\end{equation}
We use the simplest general expression for these polynomials, even if barycentric formulation \cite{WANG2025113743,berrut2004barycentric,higham2004numerical} may be exploited to increase stability. From Eq. \eqref{eq::interp1d}, multivariate Lagrange interpolation polynomials over $\bb{I}=I_0\times I_2\times I_3$ are obtained on product grids $\bb{L}^3$ using a simple tensorised expression for any $\bo{y}\in \bb{I}$, $\bo{y}^\diamond\in\bb{L}^3$:
\begin{equation}
\label{eq::interpnd}
    S_\bb{I}(\bo{y}^\diamond,\bo{y}) := S_{I_0}(y_0^\diamond,y_0)\hspace{0.1cm}S_{I_1}(y_1^\diamond,y_1)\hspace{0.1cm}S_{I_2}(y_2^\diamond,y_2),\hspace{0.5cm}\bo{y}^\diamond = (y_0^\diamond,y_1^\diamond,y_2^\diamond).
\end{equation}
\begin{remark}
For the sake of clarity, we limited this presentation to the case in which the interpolation orders are the same along each dimension. The discussion can be easily extended to arbitrary orders.
\end{remark}
Any continuous function $f$ can then be approximated by
\begin{equation}
\label{eq::fapprox}
    f(\bo{y}) \approx \sum_{\bo{y}^\diamond\in \bb{L}^3}f\circ\phi^{-1}_{\bb{I}}(\bo{y}^\diamond)S_{\bb{I}}(\bo{y}^\diamond,\bo{y})=: \underbrace{\kk{I}_{\bb{L}^3}[f](\bo{y})}_{\substack{\text{Interpolation of }f\text{ on }\bo{y}\in \bb{I}\\\text{with interpolation nodes }\bb{L}^3}},
\end{equation}
where the inverse mapping $\phi_\bb{I}^{-1}$ is defined as
\begin{equation}
    \phi^{-1}_{\bb{I}}(\bo{y}^\diamond) := (\phi^{-1}_{I_0}(y_0^\diamond),\phi^{-1}_{I_1}(y_1^\diamond),\phi^{-1}_{I_2}(y_2^\diamond)).
\end{equation}

\subsubsection{Interpolation of complex exponential}
\label{ss::intrpcomplexexpo}
Following Eq. \eqref{eq::interpnd}, we obtain for $\bo{y}\in \bb{I} = I_0\times I_2\times I_3$
\begin{equation}
    e^{2i\pi\langle\bo{m},\bo{y}\rangle} \approx \prod_{p=0}^2\sum_{y^\diamond\in \bb{L}^3}e^{2i\pi m_p\phi^{-1}_{I_p}(y^\diamond)}S_{I_p}(y^\diamond,y_p),
\end{equation}
which writes in matrix form
\begin{equation}
    \bigg[e^{2i\pi\langle\bo{m},\bo{y}\rangle}\bigg]_{\bo{m}\in \bb{M}^3} \approx \left(\bigotimes_{p=0}^2\underbrace{\bigg[e^{2i\pi m \phi^{-1}_{I_p}(y^\diamond)}\bigg]_{m\in \bb{M},y^\diamond\in \bb{L}}}_{=:\cc{F}_{\bb{M},\phi^{-1}_{I_p}(\bb{L})}}\right)\bigg[S_{\bb{I}}(\bo{y}^\diamond,\bo{y})\bigg]_{\bo{y}^\diamond\in \bb{L}^3}.
\end{equation}

This allows us to recover the target formula in Eq. \eqref{eq::targetexpformula} (i.e. a tensorised Fourier matrix). The choice of one-dimensional interpolation grid plays a significant role in the accuracy of the interpolation process. Classical one-dimensional choice (hence possibly extended to the multivariate case through tensorisation) lies in the \textit{Chebyshev nodes}, leading to small Lebesgue constant \cite{10.1007/BF01933407} on one-dimensional settings. However, many numerical method are built upon equispaced nodes \cite{pmeref,ankh} to speed up computations. Explicit convergence requirements are provided in the literature for this kind of interpolation when dealing with complex exponential, recalled here:
\begin{theorem}
\label{thm::trefethen}
\textbf{(Trefethen and Weideman, 1991 \cite{TREFETHEN1991247} Thm. 1)} Let $p_n$ be the polynomial interpolant on an equispace grid to $e^{i\eta x}$ in the points $0,1,...,n$, $\eta\in \bb{R}$, $x\in [-1,1]$, then 
\begin{equation}
    ||e^{i\eta x} - p_n(x)||_{[0,n]} \rightarrow 0\text{ as }n\rightarrow +\infty
\end{equation}
if and only if $\eta \leq \frac{\pi}{3}$, where $||\cdot ||_{[0,n]}$ denotes the supremum norm on $[0,n]$.
\end{theorem}

To ensure the convergence, we thus assume that 
\begin{equation}
\displaystyle\mathop{max}_{\bo{y}\in B}|\bo{y}-\bo{c}|\leq M^{-1},
\end{equation}
with $\bo{c}$ the center of the computational box $B$, that is $B$ has a radius $r_B$ smaller than $M^{-1}$. Indeed, if the box is not centered at zero, then each $\bo{y} = \bo{c}+\hat{\bo{y}}$ can be decomposed into the center of the box and the its relative position $\hat{\bo{y}}$. Then,
\begin{equation}
    e^{-2i\pi\langle\bo{m},\bo{y}\rangle} = \underbrace{e^{-2i\pi\langle\bo{m},\bo{c}\rangle}}_{constant}\left(\prod_{k=0}^2\underbrace{e^{-2i\pi\langle m_k,\hat{y}_k\rangle}}_{\text{1D interpolation}}\right)
\end{equation}
and only the right term has to be interpolated. This condition, according to Thm. \ref{thm::trefethen} is only sufficient but not necessary. Moreover, using a bit more restrictive condition, one may provide simple error estimates in the case of equispaced interpolation nodes.

\begin{theorem}
\label{thm::convratemultcomp}\textbf{(Convergence rate on multivariate complex exponential interpolation)}
If $|\bo{m}|_\infty\leq M$, $r_B\leq \left(\pi M\right)^{-1}$ and $L\in \bb{N}$, $L > 1$, and the interpolation grid is a cartesian grid $\Xi = \left(\frac{2r_B}{L-1}[\![0,L-1]\!]^3\right)$, then $\exists \nu\in (0,1)$ such that
\begin{equation}
    \Bigg|e^{-2i\pi\langle\bo{m},\bo{y}\rangle}-\kk{I}_{\Xi}\Big[e^{-2i\pi\langle\bo{m},\cdot\rangle}\Big]\left(\bo{y}\right)\Bigg|_\infty \leq \cc{O}\left(\nu^{L}\right).
\end{equation}
\end{theorem}

\begin{proof}
The proof is inspired by the one provided in \cite{doi:10.1137/22M1472930}. We start from the one-dimensional worst case scenario, that is when we evaluate the complex exponential on the largest mode $M$. Through classical Lagrange error estimates, for any $x\in [0,2r_B]$, one has
\begin{equation}
\begin{aligned}
 \big|e^{-2i\pi M x} -\kk{I}_{[\![0,L-1]\!]}\Big[e^{-2i\pi M \cdot}\Big]\left(x\right)\big| &\leq \frac{\big|\partial_x^{L}\left(e^{2i\pi Mx}\right)\big|}{L!}\Bigg|\prod_{l=0}^{L-1}\left(x-\frac{2r_Bl}{L-1}\right)\Bigg|\\
 \left(\text{using }h:= \frac{2r_B}{L}\right)\hspace{1cm}&\leq \frac{(2\pi |i|M)^{L}|e^{2i\pi M x}|}{L!}\frac{(2L)!h^L}{4^LL!}\\
 &\leq \left(\frac{\pi Mh}{2}\right)^L\frac{(2L)!}{L!^2} = \left(\frac{\pi Mr_B}{L}\right)^L\frac{(2L)!}{L!^2}.
 \end{aligned}
\end{equation}
Using $\nu := \pi M r_B$, one obtains
\begin{equation}
    \big|e^{-2i\pi  x} -\kk{I}_{[\![0,L-1]\!]}\Big[e^{-2i\pi M \cdot}\Big]\left(x\right)\big|\leq  \nu^L\frac{(2L)!}{L!^2L^L}
\end{equation}

Then, the last quotient may be upper bounded by a small constant:

\begin{equation}
    \frac{(2L)!}{L!^2L^L} = \prod_{k=0}^{L-1}\frac{(2L-k)(L-k)}{(L-k)^2L}=\prod_{k=0}^{L-1}\frac{(2L-k)}{L(L-k)}\leq \frac{2^L}{L!} \leq 2,
\end{equation}
leading to $\big|e^{iM x} -\kk{I}_L[e^{iM x}](x)\big| \leq  \cc{O}\left(\nu^{L}\right)$ that tends to zero if $\nu$ is smaller than $1$, which is ensured by the theorem assumptions. For any other mode $m<M$, it is easy to see that this estimate still holds, with $\nu'<\nu$, resulting in better convergence speed.

For the multivariate case, we use Thm 2.1 in \cite{multivariate_interpolation}, stating that the overall bound can be computed as a finite sum of interpolation errors in each dimension and we recover the same estimate than for the one-dimensional case.
\end{proof}

As a direct consequence, there are at least two ways for reducing the error: 
\begin{itemize}
    \item Increase the interpolation order, that could lead to numerical instability for high interpolation orders when using equispaced nodes. Indeed, beyond a certain order, rounding errors tend to render the interpolation unusable \cite{TREFETHEN1991247}. Fortunately, for double precision arithmetic as in our case, relatively small interpolation orders ($\leq 10$ in each dimension) are sufficient in practice.
    \item Reduce the size of the computational box, that paves the way to a \textit{divide-and-conquer} strategy. This approach will be presented later in this article.
\end{itemize}

Most importantly, following classical interpolation-based approach for fast summations \cite{HACKBUSCH2002129,FONG20098712,giebermann}, Eq. \eqref{eq::fapprox} applied to multivariate complex exponential gives:
\begin{equation}
    \sum_{\bo{y}\in\bb{Y}}\bigg[e^{2i\pi\langle\bo{m},\bo{y}\rangle}\bigg]_{\bo{m}\in \bb{M}^3}q(\bo{y}) \approx \left(\bigotimes_{p=0}^2 \cc{F}_{\bb{M},\phi^{-1}_{B_p}(\bb{L})}\right)\underbrace{\bigg[S_{B}(\bo{y}^\diamond,\bo{y})\bigg]_{\bo{y}^\diamond\in \bb{L}^3,\bo{y}\in \bb{Y}}}_{=:\cc{S}_{\bb{L},\bb{Y}}}\big[q(\bo{y})\big]_{\bo{y}\in \bb{Y}},
\end{equation}
for our computational box $B = B_0\times B_1\times B_2$. Notice that, provided $\bo{q}\in \bb{R}[Y]$, the conjugate is obtained as follows
\begin{equation}
\label{eq::interpkrons}
    \sum_{\bo{y}}\bigg[e^{-2i\pi\langle\bo{m},\bo{y}\rangle}\bigg]_{\bo{m}\in \bb{M}^3}q(\bo{y}) \approx \left(\bigotimes_{p=0}^2 \cc{F}_{\bb{M},\phi^{-1}_{B_p}(\bb{L})}\right)^*\cc{S}_{\bb{L},\bb{Y}}\bo{q}.
\end{equation}
These two last formulas exhibit a tensor structure thanks to the Kronecker matrix product particularly useful for fast methods.

\subsection{Compressed real Fourier matrices}
\label{ss::symfourier}
We here describe the compressed format on real number (instead of complex ones) used in our code. Since $\alpha(|\mathbf{m}|)$ is rotationally-invariant, introducing $\mathcal{D}_3$ as the rotation group preserving the cube:
\begin{equation*}
    \begin{aligned}
        H(\bo{x},\bo{y}) &= \sum_{\mathbf{m}\in \bb{Z}^3}\alpha(|\bo{m}|)e^{2i\pi \langle\bo{m},\bo{x}-\bo{y}\rangle}= \sum_{\mathbf{m}\in \left(\bb{Z}^3/\mathcal{D}_3\right)}\sum_{g\in \mathcal{D}_3(\bo{m})}\alpha(\underbrace{|g\cdot \bo{m}|)}_{=|\bo{m}|}e^{2i\pi \langle g\cdot \bo{m},\bo{x}-\bo{y}\rangle}\\
        &= \sum_{\mathbf{m}\in \left(\bb{Z}^3/\mathcal{D}_3\right)}\alpha(|\bo{m}|)\sum_{g\in \mathcal{D}_3(\bo{m})}e^{2i\pi \langle g\cdot \bo{m},\bo{x}-\bo{y}\rangle},
    \end{aligned}
\end{equation*}
where we used the notation $\bb{Z}^3/\mathcal{D}_3$ to refer to the fundamental domain of $\bb{Z}^3$ under the action $g$'s of $\mathcal{D}_3$, that is the minimal subset $S\subset \bb{Z}^3$ such that $\cc{D}_3\cdot S = \bb{Z}^3$ (we recover $\bb{Z}^3$ from $S$ by means of actions of $\cc{D}_3$), and $\cc{D}_3(\bo{m})$ for the subset of $\cc{D}_3$ whose action on $\bo{m}$ actually modify $\bo{m}$ (in order to deal with orbit with different cardinal). It is thus possible to only work on the fundamental domain $\bb{Z}^3/\mathcal{D}_3$. However, we will focus on an abelian subgroup $\mathcal{A}_3\equiv \left(\frac{\bb{Z}}{2\bb{Z}}\right)^3 < \mathcal{D}_3$ since it fits with the tensorised structure of multivariate FFTs. Notice that such exploitation of $\cc{A}_3$ already appeared in the literature \cite{simonett}. Moreover, possible actions of $\cc{A}_3$ can be represented by the matrices
\begin{equation*}
    \begin{bmatrix}
    a_0 & 0 & 0\\ 0 & a_1 & 0\\ 0 & 0 & a_2
    \end{bmatrix},\hspace{0.2cm}\text{with }a_j = \pm 1. 
\end{equation*}
We thus obtain, by splitting the exponential,
\begin{equation}
\label{eq::abelian_to_cos}
    \begin{aligned}
        \sum_{g\in \mathcal{A}_3(\bo{m})}e^{2i\pi \langle g\cdot \bo{m},\bo{x}-\bo{y}\rangle}&= \displaystyle\prod_{j=0}^2\left(\begin{cases}\displaystyle\sum_{a_j=\pm 1}e^{2i\pi a_j\cdot m_j(x_j-y_j)}&\text{if }m_j\neq 0\\\hspace{2cm}1&\text{otherwise.}\end{cases}\right).
    \end{aligned}
\end{equation}
Since, when $j\neq 0$, we have
\begin{equation*}
    \sum_{a_j=\pm 1}e^{2i\pi a_j\cdot m_j(x_j-y_j)} = e^{2i\pi  m_j(x_j-y_j)}+\overline{e^{2i\pi  m_j(x_j-y_j)}} = \underbrace{2\frak{Re}\Big\{e^{2i\pi m_j(x_j-y_j)}\big\}}_{= 2cos(2\pi m_j(x_j-y_j))},
\end{equation*}
then Eq. \eqref{eq::abelian_to_cos} becomes
\begin{equation*}
    \sum_{g\in \mathcal{A}_3(\bo{m})}e^{2i\pi \langle g\cdot \bo{m},\bo{x}-\bo{y}\rangle} = \prod_{j=0}^2\left(2^{\delta_0(m_j)}cos(2\pi m_j(x_j-y_j))\right)
\end{equation*}
with $\delta_{\bb{R}^*}(m) = 0$ if $m=0$ and $\delta_{\bb{R}^*}(m) = 1$ if $m\neq 0$. This expression is not separated in terms of $\bo{x}$ and $\bo{y}$ variables, so we exploit the following better suited equality:
\begin{equation}
\begin{aligned}
    \sum_{g\in \mathcal{A}_3(\bo{m})}e^{2i\pi \langle g\cdot \bo{m},\bo{x}-\bo{y}\rangle} = \prod_{j=0}^22^{\delta_{\bb{R}^*}(m_j)}&\Big(\underbrace{cos(2\pi m_j(x_j))}_{=:c_{m_j}(x_j)}\underbrace{cos(2\pi m_j(y_j))}_{=:c_{m_j}(y_j)}\\&+\underbrace{sin(2\pi m_j(x_j))}_{=:s_{m_j}(x_j)}\underbrace{sin(2\pi m_j(y_j))}_{=:s_{m_j}(y_j)}\Big).
\end{aligned}
\end{equation}
Written in separated form, one gets
\begin{equation}
    \sum_{g\in \mathcal{A}_3(\bo{m})}e^{2i\pi \langle g\cdot \bo{m},\bo{x}-\bo{y}\rangle} =\prod_{j=0}^22^{\delta_{\bb{R}^*}(m_j)}\underbrace{\begin{bmatrix}
    c_{m_j}(x_j) & s_{m_j}(x_j)
    \end{bmatrix}}_{=:w(m_j,x_j)^T}\begin{bmatrix}
    c_{m_j}(y_j) \\ s_{m_j}(y_j)
    \end{bmatrix}.
\end{equation}
In the general case this factorisation is only true for single particle interaction. Fortunately, this result extends to any product grid. Indeed, let $\bb{X} = \bb{X}_0\times \bb{X}_1\times \bb{X}_2$, $\bb{Y} = \bb{Y}_0\times \bb{Y}_1\times \bb{Y}_2$, the following decomposition holds
\begin{equation}
\label{eq::kronw}
\begin{aligned}
\Big[\sum_{g\in \cc{A}_3(\bo{m})}e^{2i\pi \langle g\cdot \bo{m},\bo{x}-\bo{y}\rangle}\Big]_{\bo{x}\in \bb{X},\bo{y}\in \bb{Y}} &=  \Bigg[\prod_{j=0}^22^{\delta_{\bb{R}^*}(m_j)}w(m_j,x_j)^Tw(m_j,y_j)\Bigg]_{\bo{x}\in \bb{X},\bo{y}\in \bb{Y}}\\
    &= \bigotimes_{j=0}^2\cc{W}_{m_j,\bb{X}_j}^*\cc{W}_{m_j,\bb{Y}_j}
\end{aligned}
\end{equation}
using 
\begin{equation}
    \cc{W}_{m,\bb{Y}_j} := \Bigg[\sqrt{2^{\delta_{\bb{R}^*}(m)}}w(m,y)\Bigg]_{y\in \bb{Y}_j} \in \bb{R}^{2\times |\bb{Y}_j|},
\end{equation}
These matrices are the building blocks of the structure drawn in Thm. \ref{thm::kronpme}. They are real and only consider positive Fourier modes. Hence, this reduces the storage cost for their assembly. The explicit formulas allow their \textit{on-the-fly} computation.

\subsection{Overall compression}
\label{ss::overallcomp}
It is now possible to prove Thm. \ref{thm::kronpme}. 
It has been shown in Sect. \ref{ss:decompalpha} that we can approximate $\alpha$ (see Eq. \eqref{eq::quaddecompalpham}) in such a way that
\begin{equation}
    \alpha(|\bo{m}|) \approx \left(\sum_{\lambda\in \Lambda}\omega_\lambda\prod_{j=0}^{2}\alpha_{\lambda}(m_j)\right)-c_\lambda\delta_\bo{0}(\bo{m}),
\end{equation}
where $\delta_\bo{0}(\bo{m})=\begin{cases}1\text{ if }\bo{m}=\bo{0}\\0\text{ otherwise}\end{cases}$. Let $\cc{S}_\bb{X}\in \bb{R}[\bb{G},\bb{X}]$ and $\cc{S}_\bb{Y}\in \bb{R}[\bb{H},\bb{Y}]$ be the matrices formed by Lagrange interpolation polynomials on product interpolation grids $\bb{G}=\prod_{j=0}^{2}\bb{G}_j$ (resp. $\bb{H}=\prod_{j=0}^{2}\bb{H}_j$) over the computational box $B$, i.e. 
\begin{equation}
\label{eq::sxsyfromB}
    \cc{S}_\bb{X}:=\begin{bmatrix}
    S_{B}(\bo{x}^\diamond,\bo{x})
    \end{bmatrix}_{\bo{x}^\diamond\in\bb{G},\bo{x}\in \bb{X}},\hspace{0.5cm}\cc{S}_\bb{Y}:=\begin{bmatrix}
    S_{B}(\bo{y}^\diamond,\bo{y})
    \end{bmatrix}_{\bo{y}^\diamond\in\bb{H},\bo{y}\in \bb{Y}}.
\end{equation}
If $2r_{B}<M^{-1}$, following Eq. \eqref{eq::kronw} and Eq. \eqref{eq::interpkrons}, we have
\begin{equation}
\label{ss::proofmain1}
\begin{aligned}
\Big[\sum_{\bo{m}\in\bb{M}^3}\alpha(|\bo{m}|)e^{2i\pi \langle \bo{m},\bo{x}-\bo{y}\rangle}\Big]_{\bo{x}\in \bb{X},\bo{y}\in \bb{Y}} \\\approx  \cc{S}_{\bb{X}}^*\underbrace{\left(\sum_{\bo{m}\in\bb{M}^{3}_+}\left(\sum_{\lambda\in \Lambda}\omega_\lambda\prod_{j=0}^{2}\alpha_{\lambda}(m_j)\right)\bigotimes_{j=0}^2\cc{W}_{m_j,\bb{G}_j}^*\cc{W}_{m_j,\bb{H}_j}\right)}_{=(*)}\cc{S}_{\bb{Y}}\\
-\underbrace{\sum_{\bo{m}\in\bb{M}^3} c_\Lambda\delta_\bo{0}(\bo{m})\Big[e^{2i\pi \langle \bo{m},\bo{x}-\bo{y}\rangle}\Big]_{\bo{x}\in \bb{X},\bo{y}\in \bb{Y}}}_{=(**)}
\end{aligned}
\end{equation}
where $\bb{M}_+ := [\![0,M]\!]$. We restricted ourselves to positive $\bo{m}$'s since the other modes can be eliminated using symmetries, see Sect. \ref{ss::symfourier}. 

The term $(*)$ in Eq. \eqref{ss::proofmain1} can be turned into
\begin{equation}
\begin{aligned}
    (*) &= \sum_{\lambda\in \Lambda}\omega_\lambda\bigotimes_{j=0}^2\left(\sum_{m_j\in\bb{M}_+}\alpha_{\lambda}(m_j)\cc{W}_{m_j,\bb{G}_j}^*\cc{W}_{m_j,\bb{H}_j}\right).
\end{aligned}
\end{equation}
Introducing $\cc{V}[\lambda]_{h_j}\in \bb{R}\big[\bb{M}_+\big]$ as
\begin{equation}
    \Big[\cc{V}[\lambda]_{\bb{H}_j}\Big]_{m} := \sqrt[3]{\omega_\lambda}\sqrt{\alpha_{\lambda,j}(m)}\cc{W}_{m,\bb{H}_j},
\end{equation}
we end up with
\begin{equation}
\label{eq::VtV}
\begin{aligned}
(*) &\approx  \sum_{\lambda\in \Lambda}\bigotimes_{j=0}^2\left(\cc{V}[\lambda]_{\bb{G}_j}^*\cc{V}[\lambda]_{\bb{H}_j}\right) =:  \sum_{\lambda\in \Lambda}\bigotimes_{j=0}^2\cc{A}^{(j)}_{\bb{G},\bb{H}}(\lambda).
\end{aligned}
\end{equation}
The second term (**) in Eq. \eqref{ss::proofmain1} vanishes at any non-zero $\bo{m}$, leaving us with
\begin{equation}
\begin{aligned}
    (**) =c_\Lambda\delta_\bo{0} (\bo{0}) \Big[e^{2i\pi \langle \bo{m},\bo{x}-\bo{y}\rangle}\Big]_{\bo{x}\in \bb{X},\bo{y}\in \bb{Y}} = c_\lambda\bo{1}\bo{1}^T,
\end{aligned}
\end{equation}
recalling that $\bo{1}$ denotes a column vector full of ones. Finally, combining these approximations with Eq. \eqref{ss::proofmain1} provides the matrix form drawn in Thm. \ref{thm::kronpme}
\begin{equation}
\label{eq::formfinale}
    \Bigg[\sum_{\bo{m}\in\bb{M}^3}\alpha(|\bo{m}|)e^{2i\pi \langle \bo{m},\bo{x}-\bo{y}\rangle}\Bigg]_{\bo{x}\in \bb{X},\bo{y}\in \bb{Y}}  \approx\left(\sum_{\lambda\in \Lambda}  \cc{S}_{\bb{X}}^* \left(\bigotimes_{j=0}^2\left(\cc{A}^{(j)}_{\bb{G},\bb{H}}(\lambda)\right)\right)\cc{S}_{\bb{Y}}
    \right)- c_\Lambda\bo{1}\bo{1}^T.
\end{equation}
In this expression, each approximation has an error that can be controlled (through interpolation order, quadrature accuracy or SVD truncation error).

\section{Fast Kronecker products evaluation over grids}
\label{s::fkpeog}
This section is dedicated to a fast parallel method designed to handle Kronecker products over product grids in an efficient way. This part can be read independently of the rest of the article. We detail in Sect. \ref{ss::settingkron} the context of this particular Kronecker product. Then, the usual fast Kronecker product is described in Sect. \ref{ss::fkp}. Its parallel counterpart is provided in Sect. \ref{ss::paralgokron}.
\subsection{Setting}
\label{ss::settingkron}
The main goal is to compute
\begin{equation}
    \bo{z} := \cc{A}\bo{q}
\end{equation}
for a matrix $\cc{A}$ defined as in the following. Suppose that we have a collection of matrices $\{\cc{A}^{(p)}_{i,j}\}_{(p,i,j)}$ such that $\cc{A}^{(p)}_{i,j}\in \mathbb{C}^{M_i^{(p)}\times N_j^{(p)}}$ and define
\begin{equation}
\label{eq::formegeneraledecompblockron}
    \cc{A}^{(p)} = \begin{bmatrix}
    \cc{A}^{(p)}_{0,0}&\hdots&\cc{A}^{(p)}_{0,K-1}\\
    \vdots&\ddots&\vdots\\
    \cc{A}^{(p)}_{K-1,0}&\hdots&\cc{A}^{(p)}_{K-1,K-1}
    \end{bmatrix}\in \mathbb{C}^{M^{(p)}\times N^{(p)}},
\end{equation}
where $M^{(p)} = \sum_{i}M_i^{(p)}$ and $N^{(p)} = \sum_{j}N_j^{(p)}$. We want to efficiently compute the matrix-vector product involving the matrix
\begin{equation}
\label{eq::aasproductkron}
    \cc{A} := \bigotimes_p \cc{A}^{(p)} \in \mathbb{C}^{M\times N}.
\end{equation}
In this setting, $K$ is the same for each $p$, i.e. there is the same number of submatrices $\cc{A}_{i,j}^{(p)}$'s in the decomposition of each $\cc{A}^{(p)}$. This constraint can easily be relaxed but we keep it for the sake of simplicity. Moreover, we suppose that
\begin{equation}
\label{eq::submatspliting}
    \cc{A}^{(p)}_{i,j} = \cc{U}_{i}^{(p)}\cc{V}_j^{(p)}, \hspace{0.5cm}\cc{U}_{i}^{(p)}\in \bb{C}^{M_i^{(p)}\times R_p}, \hspace{0.1cm}\cc{V}_{j}^{(p)}\in \bb{C}^{R_p\times N_j^{(p)}},
\end{equation}
where $R_p\in \bb{N}^*$ does not depend on $i$ or $j$. In other words, each submatrix $\cc{A}_{i,j}^{(p)}$ can be factorised into a fixed-size product with a left (resp. right) term only depending on $i$ (resp. $j$), see Fig. \ref{fig::decomp_mat_skp} for graphical representation.
\begin{figure}
    \centering
    \includegraphics[width=0.49\linewidth]{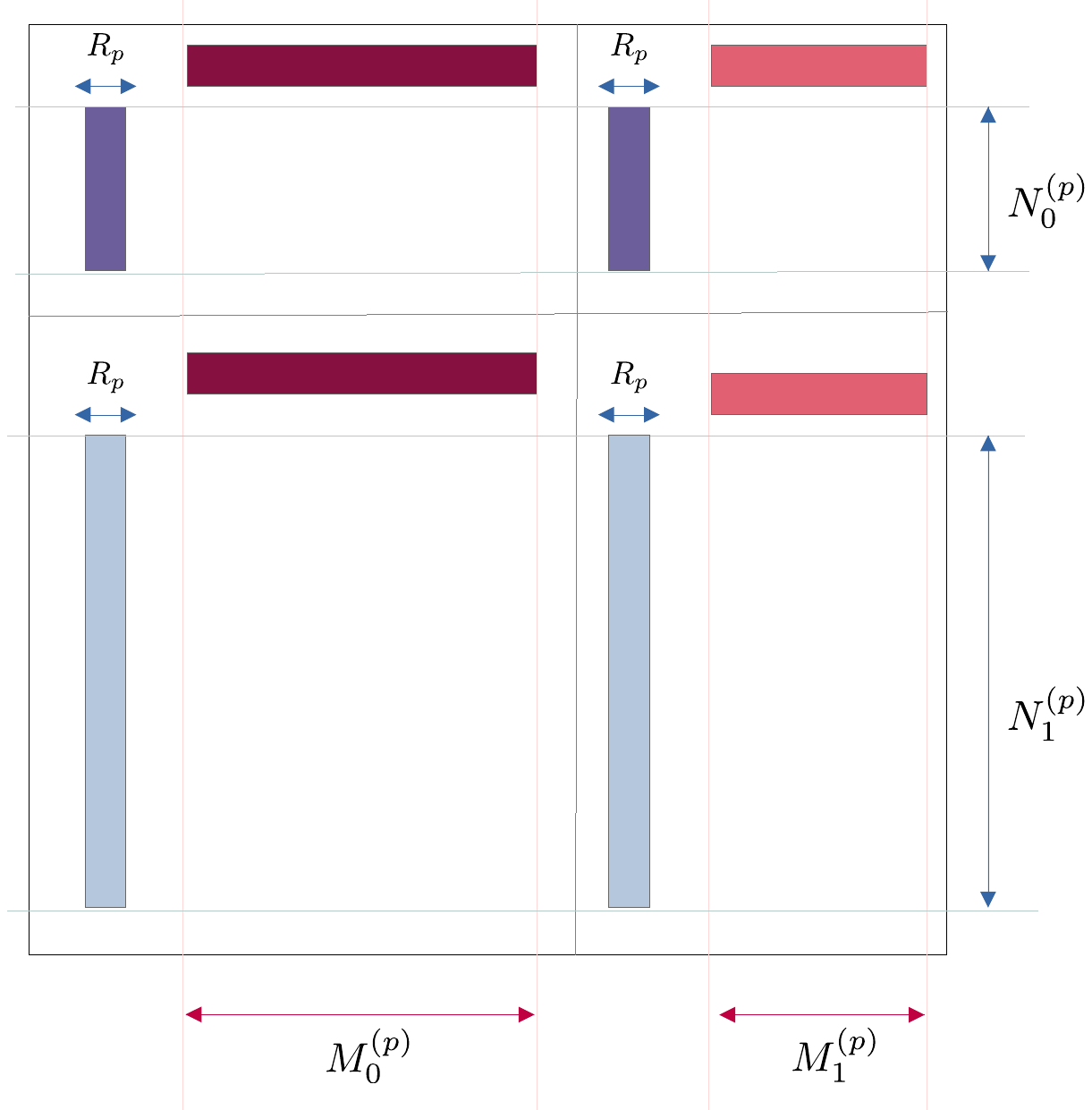}
    \includegraphics[width=0.5\linewidth]{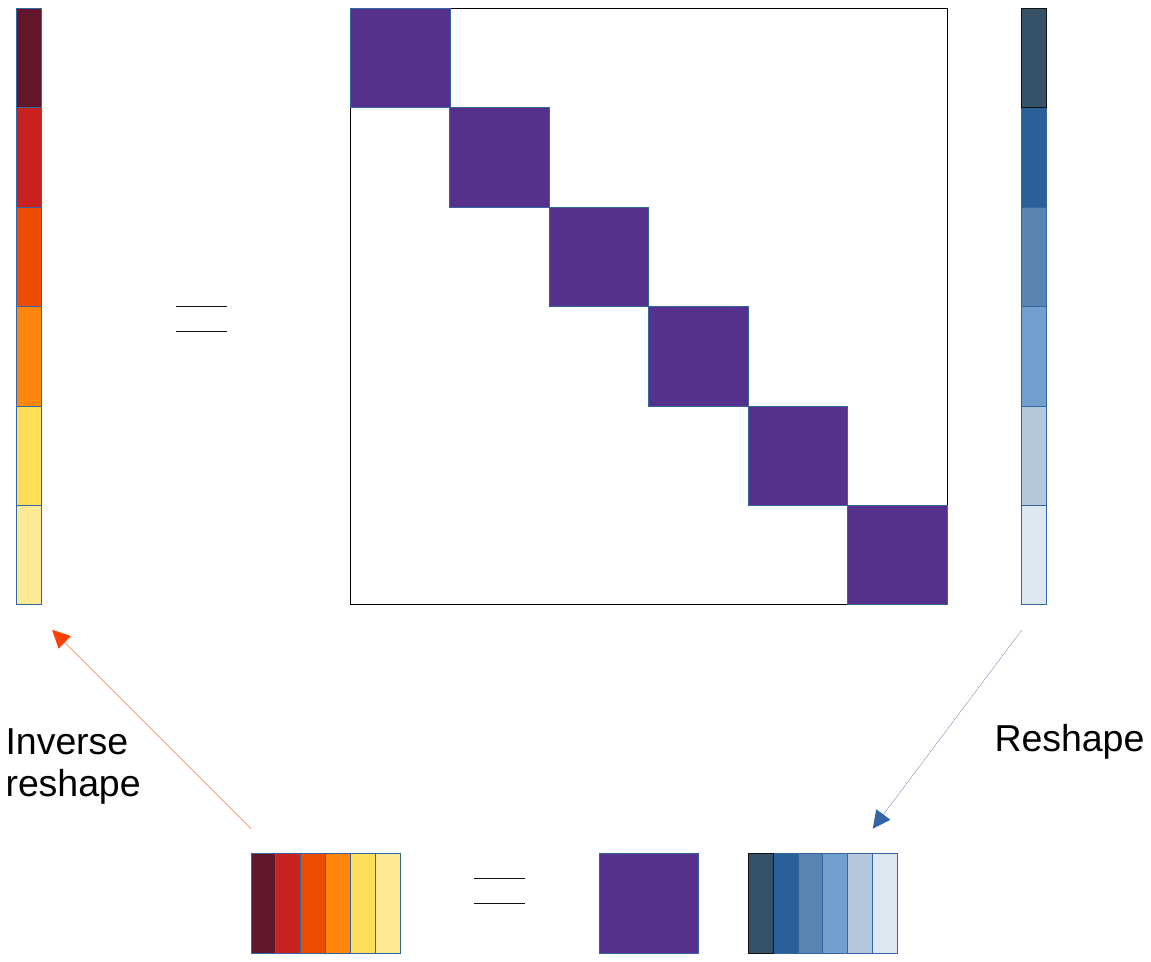}
    \caption{\textbf{(Left)} Schematic representation of a matrix $\cc{A}^{(p)}$ with $K=2$. Matrices with the same color are the same in this representation. \textbf{(Right)} Matrix and vector reshaping for matrix-matrix product.\label{fig::decomp_mat_skp}}
\end{figure}
Notice that the matrix $\cc{A}^{(p)}$ is itself a rank-$R_p$ matrix in this setting.

\subsection{Fast Kronecker product}
\label{ss::fkp}
In this section, we present the so called \textit{shuffle algorithm} \cite{shufflealgo}. By means of $\cc{A}^{(i)}\otimes \cc{A}^{(j)} = (\cc{A}^{(i)}\otimes \cc{I}_{N^{(j)}})(\cc{I}_{M^{(i)}}\otimes A^{(j)})$, with $\cc{I}_N$ the identity matrix of size $N$, the matrix in Eq. \eqref{eq::aasproductkron} can be expressed as
\begin{equation}
\label{eq::kron_as_prod}
    \begin{aligned}
        \cc{A} &= \prod_{p=0}^{d-1}\underbrace{\left(\cc{I}_{N^{(p)}_{0\rightarrow p-1}}\otimes \cc{A}^{(p)}\otimes \cc{I}_{M^{(p)}/M^{(p)}_{0\rightarrow p}}\right)}_{=: \tilde{\cc{A}}^{(p)}},
    \end{aligned}
\end{equation}
where $N^{(p)}_{0\rightarrow i} := \begin{cases}N_0^{(p)}\times ...\times N_i^{(p)}\text{ if }i>0\\1\text{ otherwise}\end{cases}$ (same for $M^{(p)}$). The product in Eq. \eqref{eq::kron_as_prod} is \textit{ordered} to preserve the dimension coherence, that is the $p^{th}$ term in the product is on the left side of the $(p+1)^{th}$.

A product by $\cc{A}$ exploiting Eq. \eqref{eq::kron_as_prod} 
thus leads to a sequence of products of matrices $\tilde{\cc{A}}^{(p)}$ for which there exists a fast application method. Indeed, there exist two permutation matrices $\cc{Q}^{left}_p$ and $\cc{Q}_p^{right}$ \cite{shufflealgo,doi:10.1137/140980326} such that
\begin{equation}
\label{eq::withpermut}
    \tilde{\cc{A}}^{(p)} = \cc{Q}^{left}_p \underbrace{\begin{bmatrix}
    \cc{A}^{(p)}&\\
    &\ddots&\\
    &&\cc{A}^{(p)}\end{bmatrix}}_{diag\left(\cc{A}^{(p)}\right)} \cc{Q}^{right}_p.
\end{equation}
The sequential algorithm directly follows, as written in Alg. \ref{alg:kronprod}.

\begin{algorithm}
\caption{Fast Kronecker matrix-vector product $\cc{A}\bo{q}$\label{alg:kronprod}}
\begin{algorithmic}
\State $\bm{\phi} = \bo{q}$
\For{$p=0,...,d-1$ in ascending order}
  \State $\bm{\phi}_{r} = \cc{Q}^{right}_{d-1-p}\bm{\phi}$
  \State $\bm{\psi} = diag\left(\cc{A}^{(d-1-p)}\right)\bm{\phi}_r$
  \State $\bm{\phi} = \cc{Q}_p^{left}\bm{\psi}$
\EndFor
\State $\bo{z}=\bm{\phi}$
\end{algorithmic}
\end{algorithm}
The permutation matrices can obviously be applied with linear complexity. Since $\cc{A}^{(p)}\in \bb{C}^{N^{(p)}\times M^{(p)}}$, and because $\cc{A}^{(p)}$ is applied to $\prod_{k=0}^{p-1}M^{(p)}\prod_{k=d-p}^{d-1}N^{(p)}$ vectors of size $M^{(p)}$ in a product by $diag\left(\cc{A}^{(p)}\right)$, this last costs 
\begin{equation}
\label{eq::complexitykroncomplete}
    \cc{O}\left(\sum_{p=0}^{d-1}\left(\prod_{k=0}^{p-1}M^{(p)}\right)\left(\prod_{k=d-p}^{d-1}N^{(p)} \right) M^{(p)}\right)
\end{equation}
where $N^{(p)}=0$ if $k\geq d$. This formula may be easier to read in simple cases. For instance, if $M^{(p)}=N^{(p)}=L$ for any $p$, then Eq. \eqref{eq::complexitykroncomplete} becomes
\begin{equation}
    \cc{O}\left(L^{d+1}\right).
\end{equation}
This complexity is better than the quadratic $\cc{O}\left(L^{2d}\right)$ of the naive Kronecker product application consisting in explicitly constructing the matrix $\cc{A}$ and in performing a direct matrix-vector product using $\cc{A}$.

\subsection{BLAS-3 optimisation and reshape}
Alg. \ref{alg:kronprod} can be numerically accelerated through the use of dense matrix-matrix operations (so called \textit{BLAS-3} operations) instead of matrix-vector ones, resulting into substantial gain in terms of arithmetic intensity (i.e. into better application timings). Indeed, a matrix-vector product involving a block-diagonal matrix $diag(\cc{A}^{(p)})\in\bb{C}[\Xi,\Xi]$, $\Xi=\Xi_0\times ... \times \Xi_{d-1}$, with \textit{the same block all along the diagonal} and a vector $\bo{q}\in \bb{C}[\Xi]$ can be written using the notations of Sect. \ref{sss::nkpa} under the following form
\begin{equation}
    diag(\cc{A}^{(p)})\cdot\bo{q} = \cc{I}^{\bullet}_p \cdot \cc{A}^{(p)} \cdot \cc{I}^{\circ}_p\cdot\bo{q},
\end{equation}
as depicted on Fig. \ref{fig::decomp_mat_skp}. This allows to write Eq. \eqref{eq::kron_as_prod} in a simple way
\begin{equation}
    \cc{A}\cdot\bo{q} = \left(\prod_{p=0}^{d-1} \cc{I}^{\bullet}_p\cdot \cc{A}^{(p)} \cdot \cc{I}^{\circ}_p\right) \cdot \bo{q},
\end{equation}
in which the permutations are hidden in the consecutive application of reshapes based on different coordinates. This provides an easier to manipulate expression of Alg. \ref{alg:kronprod} given in Alg. \ref{alg:reshapedkronprod}.
\begin{algorithm}
\caption{Reshaped Fast Kronecker matrix-vector product $\cc{A}\bo{q}$\label{alg:reshapedkronprod}}
\begin{algorithmic}
\State $\bm{\phi} = \bo{q}$
\For{$p=0..d-1$ in ascending order}
  \State $\bm{\phi}_{r} = \cc{I}^{\circ}_{d-1-p} \cdot\bm{\phi}$
  \State $\bm{\psi} = \cc{A}^{(d-1-p)}\cdot\bm{\phi}_r$
  \State $\bm{\phi} = \cc{I}^{\bullet}_{d-1-p}\cdot \bm{\psi}$
\EndFor
\State $\bo{z}=\bm{\phi}$
\end{algorithmic}
\end{algorithm}
This trick on vector stacking is a classical optimisation when dealing with precomputed matrices.

\subsection{Parallel algorithm}
\label{ss::paralgokron}
The algorithm provided in Sect. \ref{ss::fkp} can be accelerated using parallel architectures thanks to a well chosen row and column index decomposition. General algorithm for parallel Kronecker dense matrix-vector products have already been proposed \cite{tadonkiphilippe1}. Our method actually exploits close ideas, but takes into account the particular form of the matrices presented in Sect. \ref{ss::settingkron}. In more details, considering $\bo{q}\in \mathbb{C}\big[\Xi\big]$, $\bo{q}$ can be decomposed in subvectors $\bo{q}_\bo{j}$ defined as restrictions of $\bo{q}$ to product subgrids $\Xi_{\bo{j}}$ of $\Xi = \Xi_0\times ... \times \Xi_{d-1}$:
\begin{equation}
\label{eq::selectcols}
\begin{aligned}
    \Xi_p =: \cup_{k=0}^{d-1}\Xi_{p,k},\hspace{0.1cm} \bo{j}\in [\![0,K-1]\!]^d,\hspace{0.1cm}\Xi_{\bo{j}}:=\Xi_{0,j_0}\times ... \times \Xi_{d-1,j_{d-1}},\hspace{0.2cm} \bo{q}_\bo{j} := \bo{q}_{|\Xi_{\bo{j}}},\end{aligned}
\end{equation}
in which we partitioned each $\Xi_p$ into $K$ different $\Xi_{p,k}$'s. This restriction can be represented by a linear operator $\chi_\bo{j}\in\bb{R}[\Xi_{\bo{j}},\Xi]$ such that $\bo{q}_\bo{j} = \chi_\bo{j}\bo{q}$. Moreover, $\bo{q}$ is a combination of $\bo{q}_\bo{j}$'s extended by zeros (the extension operator is denoted by $\chi_\bo{j}^*$):
\begin{equation}
    \sum_{\bo{j}\in [\![0,K-1]\!]^d}\chi_\bo{j}^*\bo{q}_\bo{j} = \bo{q}.
\end{equation}
Actually, following Eq. \eqref{eq::formegeneraledecompblockron}, $\cc{A}_{\bo{i},\bo{j}} := \bigotimes_{p=0}^2\cc{A}^{(p)}_{i_p,j_p} = \chi_\bo{i}\cc{A}\chi_\bo{j}^*$ corresponds to the restriction of $\cc{A}$ to the left subgrid $\bo{i}$ and the right subgrid $\bo{j}$. Then, in these notations, the application of $\cc{A}$ to $\bo{q}$ can be seen as a sequence of Kronecker products defined on subgrids and applied to $\bo{q}_\bo{j}$'s:
\begin{equation}
    \cc{A}\cdot\bo{q} = \sum_{\bo{j}\in [\![0,K-1]\!]^d}\left(\bigotimes_{p=0}^{d-1}\begin{bmatrix}
    \cc{A}_{0,j_p}^{(p)}\\\vdots\\\cc{A}_{K-1,j_{p}}^{(p)}
    \end{bmatrix}\right)\bo{q}_\bo{j} .
\end{equation}
In other words, as already done in the literature \cite{tadonkiphilippe1}, we decompose the set of columns of $\cc{A}$ into slices with a permutation that stacks together the indices resulting from a small Kronecker product. This allows to perform in parallel a set of smaller Kronecker matrix-vector products, and to recover the final result through a simple reduction\footnote{Here, the term is used in the sense of the MPI standard, referring to a parallel sum over all involved processes.}.
In our case, things are even simpler since, according to Eq. \eqref{eq::submatspliting} and Eq. \eqref{eq::withpermut},
\begin{equation}
\label{eq::splitcolqj}
\begin{aligned}
    \cc{A}\cdot\bo{q} &=  \sum_{\bo{j}\in [\![0,K-1]\!]^d}\left(\bigotimes_{p=0}^{d-1}\begin{bmatrix}
    \cc{U}_{0}^{(p)}\\\vdots\\\cc{U}_{K-1}^{(p)}
    \end{bmatrix}\cc{V}^{(p)}_{j_p}\right)\bo{q}_{\bo{j}}
    = \sum_{\bo{j}}\prod_{p=0}^{d-1}\cc{I}^{\bullet}_p \cdot \begin{bmatrix}
    \cc{U}_{0}^{(p)}\\\vdots\\\cc{U}_{K-1}^{(p)}
    \end{bmatrix}\cc{V}^{(p)}_{j_p}\cdot \cc{I}^{\circ}_p \cdot \bo{q}_\bo{j}.
\end{aligned}
\end{equation}

Let us consider a decomposition of the overall computation on $K^d$ processes, each of them identified by a multi-index $\bo{j}\in [\![0,K-1]\!]^d$. There are $K^d$ different $\bo{q}_\bo{j}$'s and we can assign $\bo{q}_\bo{j}$ to process $\bo{j}$. The entire result $\bo{z}=\cc{A}\bo{q}$ has not to be stored on each process, and we can assign $\bo{z}_\bo{i}$ to process $\bo{i}$ where $\bo{z}_\bo{i} := \chi_\bo{i}\bo{z}$.
This turns Eq. \eqref{eq::splitcolqj} into
\begin{equation}
\label{eq::zprodsum}
\begin{aligned}
    \bo{z}_\bo{i} 
    &=\sum_{\bo{j}\in [\![0,K-1]\!]^d}\prod_{p=0}^{d-1}\left(\cc{I}^{\bullet}_p \cdot 
    \cc{U}_{i_p}^{(p)}\right)\left(\cc{V}^{(p)}_{j_p}\cdot \cc{I}^{\circ}_p \cdot \bo{q}_\bo{j}\right)\\
\end{aligned}.
\end{equation}
Since $\sum_{\bo{j}\in [\![0,K-1]\!]^d} = \sum_{j_0\in [\![0,K-1]\!]}...\sum_{j_{d-1}\in [\![0,K-1]\!]}$, one may pass the different one-dimensional sums within the product in Eq. \eqref{eq::zprodsum}. Thus, the needed reductions can be performed over $K$ processes each time (instead of $K^d$), and are performed in parallel by $K^{d-1}$ groups of processes. These communications only involve vectors of size depending on $R_p$.
\begin{algorithm}
\caption{Fast Split Parallel Kronecker matrix-vector product $\cc{A}\bo{q}$ (SPKMV($\cc{A},\bo{q}$))\label{alg:parkronprod}}
\begin{algorithmic}
\State $\bo{j} = $process index
\State $\bo{J}_p = \big\{\bo{i}\hspace{0.1cm}|\hspace{0.1cm}i_k=j_k\hspace{0.1cm}\forall\hspace{0.05cm}k\neq p\big\}$
\State $\bm{\phi}_\bo{j} = \bo{q}_\bo{j}$
\For{$k=0,...,d-1$ in ascending order}
  \State $\bm{\phi}_{r,\bo{j}} = \cc{V}^{(d-1-p)}_{j_{d-1-p}}\cc{I}^{\circ}_{d-1-p} \cdot\bm{\phi}_\bo{j}$
  \State $\bm{\psi}_\bo{j} = \sum_{\bo{i}\in \bo{J}_p}\bm{\phi}_{r,\bo{j}}$ using parallel reduction
  \State $\bm{\phi}_\bo{j} = \cc{I}^{\bullet}_{d-1-p}\cdot \cc{U}^{(d-1-p)}_{i_{d-1-p}}\bm{\psi}_\bo{j}$
\EndFor
\State $\bo{z}=\bm{\phi}$
\end{algorithmic}
\end{algorithm}
We provide in Alg. \ref{alg:parkronprod} the resulting parallel algorithm. In its definition, $\bo{J}_p$ can be considered as a \textit{communicator} involving the multi-index of processes interacting with local process during the $p^{th}$ step of the parallel algorithm. Thus, there are $d$ parallel reductions on vectors of size $R_p$ among $K$ processes in this algorithm, resulting in parallel complexity
\begin{equation}
    \cc{O}\left(log(K)\left(\alpha+\beta \mathop{max}_{p=0}^{d-1}\{R_p\}\right)\right)
\end{equation}
where $\alpha$ is the network latency and $\beta$ refers to the inverse bandwidth \cite{calu}.

\section{Kroneckerised Particle Mesh Ewald}
\label{s::kpmepar}
In Sect. \ref{s::sumofkronprod}, we demonstrated that the convolution in Eq. \eqref{eq::sumhexpr} can benefit from SKP expressions. Hence, in Sect. \ref{ss::parandinterp}, we present a \textit{divide-and-conquer} strategy to formulate Eq. \eqref{eq::sumhexpr} in a suitable way for fast computations. Then, in Sect. \ref{sec::parallel}, we present the parallel method combining this strategy with parallel Kronecker matrix-vector products presented in Sect. \ref{s::fkpeog}. Finally, in Sect. \ref{ss::complecityanalysispskppme}, we compute the actual complexity of the proposed method.

\subsection{Partitioning and interpolation}
\label{ss::parandinterp}
The main idea detailed in this section is to exploit a domain decomposition into a grid form of the particle distribution, with cells (i.e. elements of the decomposition) of constant size. Let $c$ be such a cell. According to Sect. \ref{ss::intrpcomplexexpo}, to ensure the convergence of interpolation process on $c$, denoting $r_c$ its radius, we assume that $r_c<M^{-1}$, meaning that we have a total of $\cc{O}\left(\lceil Mr_B\rceil\right)$ cells in each direction, resulting in a total of $K^3=\cc{O}\left(\lceil Mr_B\rceil^3\right)$ cells. Since we are dealing with a three-dimensional cartesian cell grid, each cell can be localised in this grid thanks to a multi-index in $\bo{i}\in [\![0,K-1]\!]^3$. This results in a matrix decomposition of $\cc{H}_M$ (see Eq. \eqref{eq::sumhexpr}) of the form
\begin{equation}
    \begin{bmatrix}
    \cc{A}_{c_\bo{0},c_\bo{0}}&\hdots &\cc{A}_{c_\bo{0},c_{(K-1)\bo{1}}}\\
    \vdots & \ddots & \vdots\\
    \cc{A}_{c_{(K-1)\bo{1}},c_\bo{0}}&\hdots &\cc{A}_{c_{(K-1)\bo{1}},c_{(K-1)\bo{1}}}
    \end{bmatrix},\hspace{0.1cm}\cc{A}_{c_\bo{i},c_\bo{j}}\left(\bo{x},\bo{y}\right) := \sum_{0\neq \bo{m}\in \bb{Z}^3}\alpha(|\bo{m}|)e^{2i\pi \langle \bo{x}-\bo{y},\bo{m}\rangle},
\end{equation}
for any $\bo{x}\in c_\bo{i}$ and $\bo{y}\in c_\bo{j}$, where $\cc{A}_{c_\bo{i},c_\bo{j}}$ is the matrix of interactions between cells $c_\bo{i}$ and $c_\bo{j}$, with $K$ cells per dimension (i.e. a grid of $K^3$ cells). 

Moreover, fixing an interpolation order $L$ supposed to be the same in each direction for each cell, one can provide a product interpolation grid on each of them. In this setting, the interpolation grid $\bb{G}_\bo{i}$ on the $\bo{i}^{th}$ cell, writes
\begin{equation}
    \bb{G}_\bo{i} = \bb{G}_{i_0}\times \bb{G}_{i_1}\times \bb{G}_{i_2}, \hspace{0.5cm} c_{\bo{i}}= Conv(\bb{G}), 
\end{equation}
where $Conv(\bb{G})$ denotes the convex hull of $\bb{G}$, that is a cuboid in our case. Thus, as a consequence of Eq. \eqref{eq::formfinale}, denoting $\cc{A}^{(j)}_{\bo{k},\bo{l}}(\lambda) := \cc{V}[\lambda]_{\bb{G}_{k_j}}^*\cc{V}[\lambda]_{\bb{G}_{l_j}}$ (see Eq. \eqref{eq::VtV}) and $\cc{S} = \begin{bmatrix}
\cc{S}_{c_\bo{0}} & &\\
& \ddots &\\
&& \cc{S}_{c_{(K-1)\bo{1}}}
\end{bmatrix}$ the diagonal matrix of interpolation polynomials over each cell (i.e. as in Eq. \eqref{eq::sxsyfromB} but on each cell instead of th entire $B$), $\cc{H}_M$ can be approximated (see Eq. \eqref{eq::formfinale}) by
\begin{equation}
\label{eq::jesaisplusquoimettre}
\begin{aligned}
\cc{S}^*\left(
\sum_{\lambda\in \Lambda}\bigotimes_{j=0}^2
\underbrace{\begin{bmatrix}\cc{A}_{\bo{0},\bo{0}}^{(j)}&\hdots &\cc{A}_{\bo{0},(K-1)\bf{1}}^{(j)}\\\vdots&\ddots & \vdots\\\cc{A}_{(K-1)\bf{1},\bo{0}}^{(j)}&\hdots &\cc{A}_{(K-1)\bf{1},(K-1)\bf{1}}^{(j)}\end{bmatrix}(\lambda)}_{\cc{A}^{(j)}(\lambda)}\right)\cc{S}
    -c_\Lambda\bo{1}\bo{1}^T,
    \end{aligned}
\end{equation}
which compactly rewrites
\begin{equation}
    \left(\cc{S}^*\left(\sum_\lambda\underbrace{\bigotimes_{j=0}^2\cc{A}^{(j)}(\lambda)}_{\cc{A}(\lambda)} \right)\cc{S}\right)-c_\Lambda\bo{1}\bo{1}^T,
\end{equation}
which itself has the same form than in Thm. \ref{thm::kronpme}. Here, each $\cc{A}(\lambda)$ verifies the assumptions of Sect. \ref{ss::settingkron}, meaning that we can perform fast products using (parallel) split Kronecker matrix-vector products (SPKMV) as in Sect. \ref{ss::paralgokron}. We have an amount of SPKMV equal to the number $|\Lambda|$ of terms in the SKP. 

The handling of the second term of Eq. \eqref{eq::jesaisplusquoimettre} is easier since we have
\begin{equation}
\label{eq::termbonsum}
    c_\Lambda\bo{1}\bo{1}^T\bo{q} = c_\Lambda\left(\sum_{j}q_j\right)\bo{1}=:c(\Lambda,\bo{q})\bo{1},
\end{equation}
$c(\Lambda,\bo{q})\bo{1}\in \bb{R}[\bb{X}]$ referring to the \textit{correction term of} $\bo{q}$. This correction thus takes the form of a constant added to each term after the sequence of SPKMV.

\subsection{Parallel algorithm}
\label{sec::parallel}
It is now possible to provide the global parallel algorithm allowing to compute Eq. \eqref{eq::sumhexpr} based on a SKP approximation of the reciprocal kernel matrix, as proposed in Alg. \ref{alg:kpme}. For the sake of simplicity, we assume that
\begin{itemize}
    \item The source point cloud $\bb{Y}$ is the same than the target one $\bb{X}$,
    \item If a particle lies at the boundaries between multiple cells, it is arbitrary assigned to only one of them,
    \item All the interpolation orders in each interpolation grid in each direction are the same (this reduces the total number of parameters),
    \item The decomposition is such that each cell has a radius less than $M^{-1}$,
    \item Each MPI process stores the input and output corresponding to its multi-index.
\end{itemize}

\begin{algorithm}
\caption{Kroneckerised Particle Mesh Ewald (KPME)\label{alg:kpme}}
\begin{algorithmic}
\State $\bo{j} = $process multi-index
\State $\bo{W}$ the set of all process multi-indices
\State Decompose and distribute the point cloud according to Eq. \eqref{eq::selectcols} among $\bo{W}$
\State $\nu_\bo{j} = $ the sum of $\bo{q}_\bo{j}$'s elements
\State $c(\Lambda,\bo{q}) = \sum_{\bo{i}\in \bo{W}}\nu_\bo{i}$ through reduction over $\bo{W}$
\State $\bm{\phi}_\bo{j} = \cc{S}\cdot\bo{q}_\bo{j}$ apply the interpolation matrix
\For{$\lambda\in \Lambda$}
    \State $\bm{\phi}_\bo{j} \leftarrow \bm{\phi}_\bo{j} + SPKMV(\cc{A}(\lambda),\bo{\psi})$ in parallel using Alg. \ref{alg:parkronprod}
\EndFor
\State $\bm{\psi}_\bo{j} = \cc{S}^*\cdot\bm{\phi}_\bo{j}$ using the transpose of the interpolation matrix
\State $\bo{z}_\bo{j} = \bm{\psi}_\bo{j}  - c(\Lambda,\bo{q})\bo{1}$ to apply the correction term and obtain the final result.
\end{algorithmic}
\end{algorithm}
One can summarise Alg. \ref{alg:kpme} as in the following sequence of parallel steps
\begin{itemize}
    \item The algorithm starts by interpolating local particles in each process,
    \item For each $\lambda$, i.e. for each term in the SKP decomposition of $\cc{H}_M$, we compute the SPKMV between the SKP term $\cc{A}(\lambda)$, distributed among processes,
    \item Once the sum of SKP term contributions is obtained, we interpolate back,
    \item The computation ends by correcting the result because of the constant term in SKP decomposition (this term being obtained by simple reduction among all processes according to Eq. \eqref{eq::termbonsum}).
\end{itemize}

A graphical interpretation of Alg. \ref{alg:kpme} is provided in Fig. \ref{fig::kpmealgo}. In this algorithm, the point cloud decomposition is the same than a particle sorting inside a cartesian cell grid. Hence, it can be performed in a number of operation proportional to the size $N$ of the point cloud. A \textit{all2allv}-type of MPI communication is needed to distribute the particles over the MPI processes. We will not consider this particle assignment in the cost estimates, since it has to be performed only once before the simulation and can be handled through external routines that could call KPME in applications (for instance \cite{lagardere2018tinker}).

\begin{figure}
    \centering
    \includegraphics[width=\linewidth]{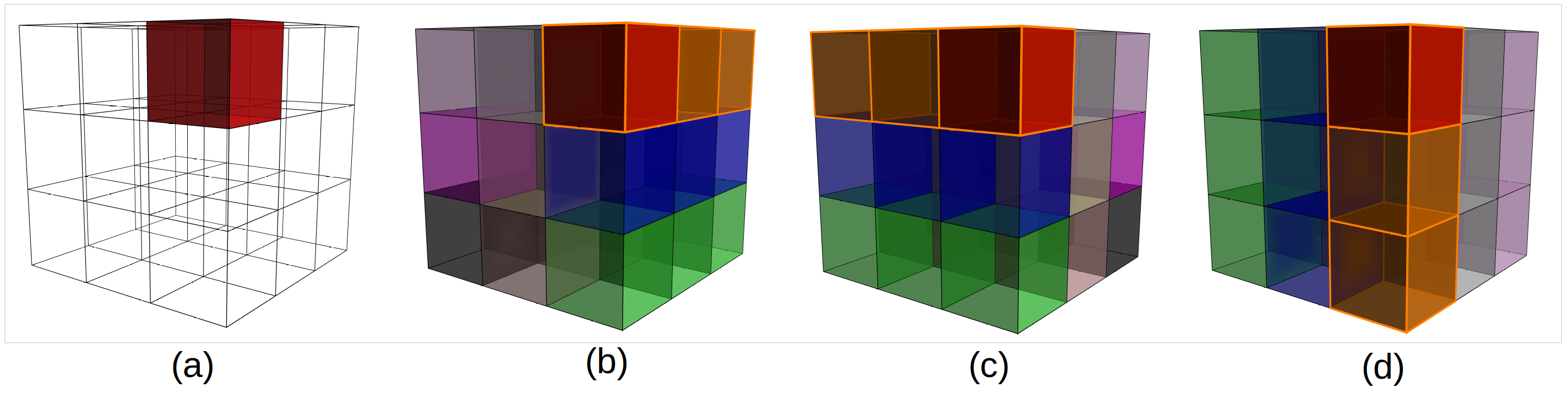}
    \caption{\label{fig::kpmealgo}Schematic view of the parallel algorithm at a given quadrature node $\lambda$ on a $3\times 3\times 3$ MPI grid. Elements of an MPI communicator are depicted in the same color (except that the red current process $\bo{l}$ which is in the orange communicator). \textbf{(a)} \cgrey{Each process independently applies the Fourier transform $\cc{V}[\lambda]_{c_{l_2}}$ to its data reshape by $\cc{I}^{\bullet}_2$ on the $Z$-axis of the $\lambda$'s term in SKP decomposition.} \textbf{(b)} The information is reduced through the $Z$-axis communicators (in the Fourier domain). \cgrey{Then the Fourier information is locally transformed back into the real one along the $Z$-axis using $\cc{V}[\lambda]_{c_{l_2}}^*$. The information (after a reshape) along the $Y$-axis is transformed into the Fourier domain using $\cc{V}[\lambda]_{c_{l_1}}$.} \textbf{(c)} The information is reduced through the $Y$-axis communicators in the Fourier domain. \cgrey{Then the Fourier information is locally transformed back into the real one along the $X$-axis using $\cc{V}[\lambda]_{c_{l_1}}^*$. The information (after a reshape) along the $X$-axis is transformed into the Fourier domain using $\cc{V}[\lambda]_{c_{l_0}}$.} \textbf{(d)} The information is reduced through the $X$-axis communicators in the Fourier domain. \cgrey{Then the Fourier information is locally transformed back into the real one along the $Z$-axis using $\cc{V}[\lambda]_{c_{l_0}}^*$.}}
\end{figure}

\subsection{Complexity analysis}
\label{ss::complecityanalysispskppme}
The analysis of Alg. \ref{alg:kpme} is quite obvious thanks to the complexity estimate of SPKMV. Indeed, the parallel complexity of Alg. \ref{alg:kpme} is simply multiplied by the number $|\Lambda|$ of terms in the SKP decomposition plus one (for the correction term assembly). This complexity is thus estimated on a grid of $K^3$ MPI processes by
\begin{equation}
\label{eq::estimatepar}
    \cc{O}\left(|\Lambda|log(K)\left(\alpha+\beta M\right)\right).
\end{equation}
The $M$ in Eq. \ref{eq::estimatepar} comes from our chosen communications that only act on \textit{slices} of the retained set of Fourier modes $\bb{M}_+$ (see Sect. \ref{ss::overallcomp}). It is worth noticing that an alternative algorithm could perform every local SKP term contribution first and only perform three last reductions grouping all SKP terms to recover the final result (omitting the correction term). This would drastically reduce the number of communications but also would increase proportionally the length of the messages (i.e. the factor in front of $\beta$).

In the case of the kernel $\cc{H}$ in Eq. \eqref{eq::hexpr}, combining Eq. \eqref{eq::estimatepar} with the lower bound in Eq. \ref{eq::estimatequad}, the complexity estimation becomes
\begin{equation}
    \Omega\left(\left(\frac{log\left(\frac{\epsilon}{16}\right)}{\pi}\right)^2log(K)\left(\alpha+\beta M\right)\right).
\end{equation}

Now, locally, denoting by $\gamma$ the cost of an operation, each process performs
\begin{equation}
    \cc{O}\left(\gamma\left(\underbrace{L^3P}_{\text{interpolations}} + \underbrace{ML^3}_{\substack{\text{product by }\\\cc{V}[\lambda]_{c_{l_j}}\text{'s}}}\right)\right)
\end{equation}
operations since the one-dimensional interpolation order of local three-dimensional interpolation grids is supposed to be the same and equal to $L$. The term $P$ here refers to the maximal number of particles in an MPI process, estimated by $\bb{Y}/K^3$ in practice (indeed, for molecular dynamics applications, atoms are quasi-uniformly distributed in the computational box). Actually, one can reduce the interpolation complexity to $\cc{O}\left(\gamma (LP+L^3)\right)$ thanks to a careful ordering of operations. The correction term is applied in linear time w.r.t. the local number of particles.

Thus, the final complexity estimate writes
\begin{equation}
\label{eq::sequentialcomplexitytotalt}
    \cc{O}\left(|\Lambda| log(K)(\alpha+\beta M) + \gamma L(P+|\Lambda| ML^2)\right).
\end{equation}
The comparison between this estimate and the PME complexity is difficult since we need further assumptions (on links between number of particles and choice of $M$), and because it requires to come back to a sequential case. However, we can assume that the number of particle in a process is proportional to the number of considered Fourier modes for the convergence to be reached. Thus, assuming that each task of MPI processes is sequentially executed and that $|\Lambda|$ is a constant (that may not be the case for small targeted precision), we obtain a coarse estimate of $\cc{O}\left(M^4\right)$ for sequential execution of the proposed algorithm, which is not asymptotically competitive with the linearithmic complexity $\cc{O}\left(M^3log(M)\right)$ of PME. Notice that we can further accelerate our method by exploiting Non-Uniform Fast Fourier Transforms (NUFFT) \cite{LEE20051}, leading to complexity estimates close to the linearithmic one. The parallel comparison requires measurements of optimised code execution timings, which is out of the scope of this article.

\section{Numerical results}
\label{s::numericals}
This section is dedicated to numerical tests on proposed KPME. The experiments in this work were carried out on MAGI, the experimental platform of University of Sorbonne Paris North (USPN) dedicated to research. This platform offers researchers at the institution High-Performance Computing (HPC), cloud and storage services.\footnote{https://github.com/Nyk0/magi-wiki/tree/main}.

Our code is written in C++ and exploits MPI routines (using openmpi-5.0.3). Notice that we compiled our code with option -march=x86-64. Better timings can probably be achieved by considering architecture specific optimisations. We also used -O3 -ffast-math -mfma -fno-trapping-math flags.

\subsection{Convergence}
First, we check the practical convergence of our approximation scheme. In Fig. \ref{fig::convergencetotal}, we report the error plots obtained on a single cell using interpolation only and with all approximations (including quadrature based on Yarvin-Rokhlin rules \cite{yarvinrokhlin} and compression of $\bm{\alpha}_M$). 
\begin{figure}
    \centering
    \includegraphics[width=0.45\linewidth]{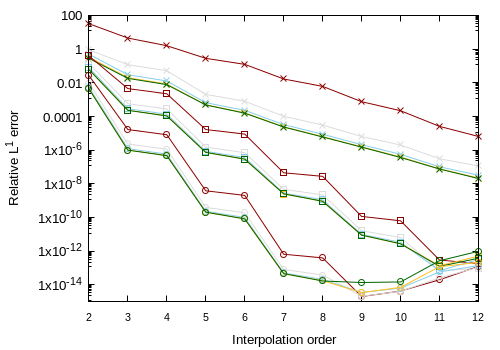}
    \includegraphics[width=0.45\linewidth]{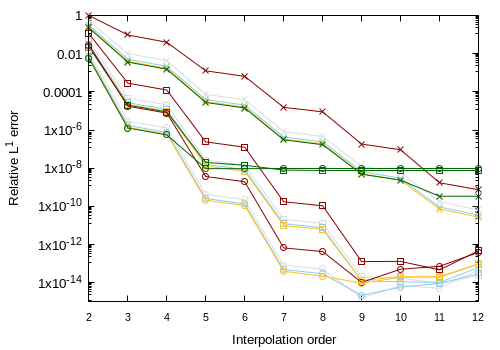}
    \caption{Convergence plot for various $M$ and cell size ratio $\nu$ (see proof of Thm. \ref{thm::convratemultcomp}) using the following code. Red: $M=2$; Grey: $M=4$; Blue: $M=8$; Yellow: $M=16$; Green: $M=32$; Crosses: $\nu=2^{-1}$; Squares: $\nu=4^{-1}$; Circles: $\nu=8^{-1}$ \textbf{(Left)} Relative error on a cell w.r.t. one-dimensional interpolation order for various values of $M$ \textbf{(Right)} Relative error on final algorithm (interpolation + quadrature using Yarvin-Rokhlin rule for $M=12$ + split compressed real Fourier matrices)\label{fig::convergencetotal}}
\end{figure}
A direct observation shows that the overall error fits with the interpolation one (provided that the condition of Thm. \ref{thm::trefethen} is fulfilled). Moreover, this interpolation error practically behaves as estimated in Thm. \ref{thm::convratemultcomp}, that is as $\cc{O}\left(\nu^L\right)$ where $\nu\approx Mr_c$. According to this figure, a precision can be approximately attained for a given cell with fixed interpolation parameter through dividing this cell into two smaller ones in a given direction as well as the interpolation order in this direction. It can be seen on the sub-figure combining all approximation that the convergence is preserved for $M$'s in the range covered by the chosen quadrature rule (i.e. $\leq 12$ in this example). This validates the criterion combination exploited in our method. Numerical instabilities, certainly the result of rounding errors caused by the use of interpolation on equispaced nodes, can also be observed when reaching $\approx 10^{-14}$ relative error (on both sub-figures).

\subsection{Performance}
\subsubsection*{Overall complexity from single cell timings} We first look at sequential performance. These last are measured on a single cell, but extrapolated to sequential multicell case by a simple multiplication by the theoretical number of cells in Fig. \ref{fig:singlecell} top. Indeed, our code was designed for parallel applications and only handles a single cell per MPI process.
\begin{figure}
    \centering
    \includegraphics[width=\linewidth]{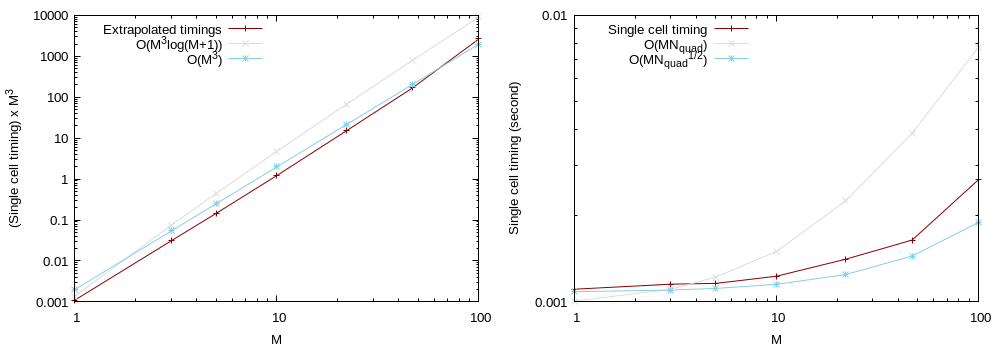}
    \includegraphics[width=\linewidth]{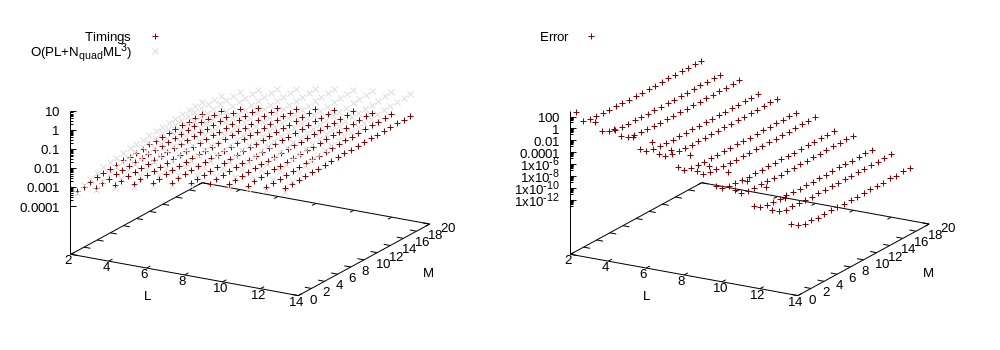}
    \caption{This test ran on AMD EPYC 9655 96-Core Processor. \textbf{[Top]} Timings of single cell execution for $10^3$ particles w.r.t. $M$ \textbf{(left)}. The sequential timings \textbf{(right)} are extrapolated by multiplying single cell timing by the total number of required cells. One-dimensional interpolation order is fixed to $L=5$. \textbf{[Bottom]} \textbf{(Left)} Single cell execution timings (seconds) w.r.t. $M$ and $L$ (red) and theoretical estimates (grey). The cell contains $10^3$ particles. \textbf{(Right)} Relative error for each corresponding $(M,L)$ pair.\label{fig:singlecell}}
\end{figure}
In this example, the sequential complexity seems to better behave than in our estimates of Sect. \ref{ss::complecityanalysispskppme}, following a kind of linearithmic curve. However, these results are extrapolated so a precise measurement on specific code for this sequential case is still needed. Nevertheless, these results also provide hint on local amount of work needed in each MPI process for varying $M$, that numerically almost better behaves w.r.t. $M$ than our estimates in Eq. \ref{eq::sequentialcomplexitytotalt} (for fixed interpolation order $L$ and number of particle $P$ in this cell). Notice that $|\Lambda|$ is not a constant when considering cardinal sine quadrature rule, hence the number of quadrature nodes $N_{quad}$ has to be taken into account in the figure. This validates the fast approach adopted for local computations.

\subsubsection*{Sensibility to parameters and induced error} We depicted on Fig. \ref{fig:singlecell} bottom the timings of code execution for a single cell while varying the number of Fourier modes as well as the interpolation order. Comparing to our theoretical estimate (with $N_{quad}$ the number of quadrature points since we exploit cardinal sine quadrature here) as in Sect. \ref{ss::complecityanalysispskppme}, we note that our estimates are slightly pessimistic, although they reflect the trend in our measured times. The corresponding error, as expected, depend mostly on the interpolation order. There still is a small dependency on the number of Fourier modes $M$, that could be a consequence of the restrictive condition (see Sect. \ref{ss::intrpcomplexexpo}) we used for complex exponential interpolation w.r.t. $M$.

\subsubsection*{Parallel tests} Regarding parallel timings, things are also tricky to evaluate. We have multiple ways of decomposing the computational box into cuboid cells: split in one, two or three directions. In addition, the type of architecture matters: MPI processes can be both on distant nodes (i.e. distributed memory) or on the same node (i.e. shared memory). We provide in Fig. \ref{fig::strongscalefinal} a strong scaling test over up to 30 nodes with various number of MPI processes in each node. In addition, we compare the different domain decomposition possibilities.
\begin{figure}
    \centering
    \includegraphics[width=0.45\linewidth]{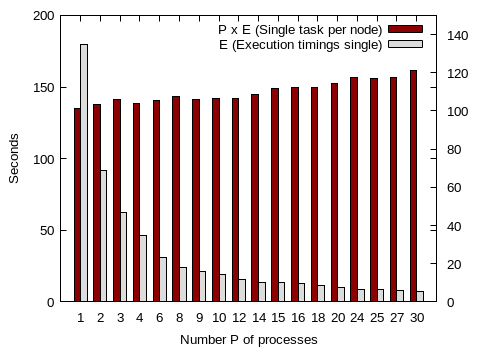}
    \includegraphics[width=0.45\linewidth]{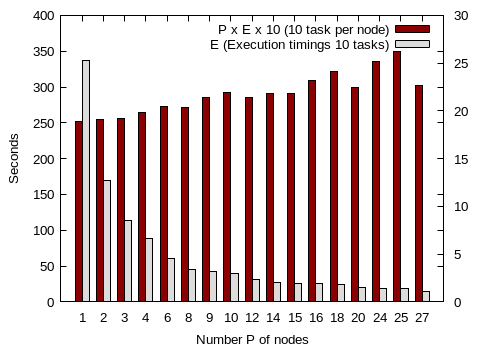}
    \includegraphics[width=\linewidth]{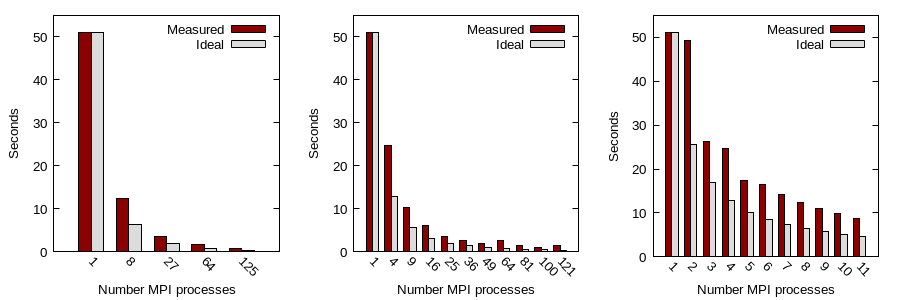}
\caption{\textbf{[Top]} Strong scaling on nodes equipped with 2 Intel(p) Xeon(p) CPU E5-2650 v3 (2.30GHz). Test case corresponds to a cubic box with $10^7$ particles, with 1D interpolation order equal to 8 and $M=10$. \textbf{(Left)} Strong scaling on nodes with a single MPI process per node. In red, number of MPI process multiplied by the measured timings (in seconds) using the left y-axis scale. In grey, the measured timings in seconds, using the right y-axis scale. \textbf{(Right)} Same format but with a total of 10 MPI processes per node. \textbf{[Bottom]} Execution timings on AMD EPYC 9554 64-Core using various types of domain decomposition \textbf{(Left)} three-dimensional decomposition with equal number of MPI processes in each direction \textbf{(Middle)} two-dimensional decomposition with equal number of MPI processes in directions $Y$ and $Z$ \textbf{(Right)} One dimensional decomposition over $Z$.\label{fig::strongscalefinal}}
\end{figure}
It appears that most of the case, our code has a $\approx 2$ factor ratio with ideal scalable timings. This results in convincing efficiency (sequential timing divided by parallel one and number of processes) of $0.83$ on $30$ nodes (and of course greater on smaller number of nodes). The strong scalability property appears less clearly on distributed+shared memory case than in the pure distributed memory one. This can be a consequence of non-optimised code for shared memory parallelism since we designed our implementation for the distributed memory case.

\subsubsection*{Effect of domain decomposition} On the distributed+shared case, one may already observe the effect of 3D decomposition compared to 2D ones, for instance with $25=1\times 5\times 5$ nodes (i.e. 2D decomposition) and $27=3\times 3\times 3$ nodes (i.e. 3D decomposition). Indeed, the case $27$ processes ran faster than the smaller one with $25$ processes, suggesting that the 3D decomposition should be more efficient. By considering greater decomposition in terms of dimensions, one linearly increase the number of messages but decreases geometrically the size of these messages. This seems to matter above a certain MPI world size (here identified as $\approx 100$). However, each decomposition type actually leads to scalable results, as illustrated on the bottom sub-figures. This shows the flexibility of our method with respect to the types of domain decomposition.

\section{Conclusion}
In this paper, we provided a new way of representing the reciprocal part of Ewald summation based on kroneckerisation of diagonal matrix involved in the Fourier domain and on particle interpolation on grids. Hence, this approach is a particle-mesh one. We proposed a parallel algorithm for this formulation based on fast Kronecker products combined with a \textit{divide-and-conquer} strategy whose parallel efficiency has been illustrated through numerical experiments. The limitations of the presented method lie in its theoretical sequential complexity, which could be improved by considering the use of NUFFT, as well as in the number of messages sent, which can also be drastically reduced to only three by grouping messages corresponding to different SKP terms. This could also be combined with an additional trivial level of parallelisation coming from the independence of terms in the SKP to possibly further reduce the communication cost. Nevertheless, we exhibited strong scaling on different configurations as well as a convergence of the overall approximation scheme in accordance with the theoretical estimates provided in this paper.

In order to be entirely convincing, the results presented should be extended to larger architectures in terms of the number of nodes. Considering the GPU implementation of the approach presented here also seems a reasonable prospect in view of the targeted applications. A careful branching with existing molecular dynamics codes with which our domain decomposition strategy is adapted (such as TINKER-HP \cite{lagardere2018tinker}) could also be a way of performing precise comparison with PME.

\section*{Acknowledgement}The author wish to thank Mi-Song Dupuy for suggesting us the bounds on optimal quadrature rules used in this article and Nicolas Greneche for his precious help on MAGI platform.

\bibliographystyle{unsrt}
\bibliography{references}

\end{document}